\documentclass[12pt]{amsart}
\usepackage[utf8]{inputenc}
\usepackage[T1]{fontenc}

\usepackage{CormorantGaramond} %%font

\usepackage[margin=1in]{geometry}
\usepackage{graphicx} 
\usepackage{hyphenat} 
\usepackage{amsmath,amsfonts,amssymb,amsthm,epsfig,epstopdf,url,array,mathtools,faktor,paralist, xfrac,mathrsfs,pifont,yfonts,mathabx,leftindex,adjustbox,stmaryrd, mathdots,booktabs,tcolorbox,float,xcolor}
\usepackage{ytableau}
\usepackage{enumitem}
\usepackage{microtype}
\usepackage[colorlinks=true,linkcolor=red!60!black,citecolor=green!60!black]{hyperref}

\usepackage{tikz}
\usepackage{tikz-cd}
\usetikzlibrary{arrows,positioning,scopes,trees,matrix,fit,backgrounds,shapes,fit,calc}

\ytableausetup{smalltableaux,centertableaux}

\theoremstyle{plain}
\newtheorem{theorem}{Theorem}[section]
\newtheorem{prop}[theorem]{Proposition}
\newtheorem{lem}[theorem]{Lemma}
\newtheorem{cor}[theorem]{Corollary}
\newtheorem{conj}[theorem]{Conjecture}
\theoremstyle{definition}
\newtheorem{defn}[theorem]{Definition}
\newtheorem{example}[theorem]{Example}
\newtheorem{remark}[theorem]{Remark}

\newtheorem{thmx}{Theorem}            % counter "thmx"
\newtheorem{propx}[thmx]{Proposition} % share the counter

\newtheorem*{thm*}{Theorem}
\newtheorem*{prop*}{Proposition}
\newtheorem*{ques*}{Question}

\newcommand{\ZZ}{\mathbb{Z}}

\newcommand{\lam}{\lambda}
\newcommand{\Kos}[2]{K_{#1,\,#2}}
\newcommand{\dom}{\trianglerighteq}
\newcommand{\bd}{\mathbf}

\newcommand{\SI}{\operatorname{SI}}

\newcommand{\Gr}{\operatorname{Gr}}

\DeclareMathOperator{\SSYT}{SSYT}

\DeclareMathOperator{\rep}{\mathrm{rep}}
\DeclareMathOperator{\GL}{\mathrm{GL}}
\DeclareMathOperator{\SL}{\mathrm{SL}}

\title{On a conjecture of King-Tollu-Toumazet for rectangular Kostka coefficients}

\author{Deepanshu Prasad}
\address{School of Mathematics and Statistics, Carleton University, Ottawa, Canada}
\email{deepanshu.prasad@gmail.com}

\begin{document}

\begin{abstract}
A conjecture of King, Tollu and Toumazet predicts the behavior of small Littlewood--Richardson
coefficients under stretching: if $c^{\nu}_{\lambda,\mu}=3$, then
$c^{M\nu}_{M\lambda,M\mu}$ equals $2M+1$ or $\binom{M+2}{2}$ for every positive integers $M$. This corresponds to the Hilbert
functions of a conic, or of the projective plane. We prove their conjecture for all rectangular Kostka coefficients $K_{\lam,\bd a}$. We realize them as the weight multiplicities of the homogeneous coordinate ring of the Grassmannian $\Gr(k,n)$, or equivalently the weight spaces of
semi-invariants of the $n$-subspace quiver with dimension vector $\beta=(k,1,1,\dotsc,1)$. For a weight $\sigma$, we show that if $\beta$ is $\sigma$-stable, then the dimension of the corresponding weight space is at least $n-2$ when $k=2$ and at
least $4$ for every Grassmannian other than $\Gr(2,4)$, $\Gr(2,5)$ and $\Gr(3,5)$. These three cases are precisely when the associated moduli space of $\sigma$-stable representations has dimension at most two. Via $\sigma$-stable decomposition, the two alternatives are determined as follows: $2M+1$ for $\Gr(2,4)$, and $\binom{M+2}{2}$ for $\Gr(2,5)$ and $\Gr(3,5)$.
\end{abstract}

\maketitle

\section{Introduction}\label{sec:intro}

Let $\mathbb{K}$ be an algebraically closed field of characteristic zero and $Q$ a finite, connected, acyclic quiver with vertex set $Q_0$ and arrow set $Q_1$. For a dimension vector $\beta\in\mathbb{N}^{\lvert Q_0\rvert}$, the group $\GL_\beta:=\prod_{x\in Q_0}\GL_{\beta(x)}$ acts on the representation space $\rep_\beta(Q)$ by $g\cdot V=(g_{ha}V(a)g_{ta}^{-1})_{a\in Q_1}$, and the ring of \emph{semi-invariants} is
\[
    \SI(Q,\beta)\coloneqq\mathbb{K}[\rep_\beta(Q)]^{\SL_\beta},\qquad \SL_\beta:=\prod_{x\in Q_0}\SL_{\beta(x)}.
\]
Every integral weight $\sigma\in\mathbb{Z}^{\lvert Q_0\rvert}$ defines a multiplicative character $\chi_\sigma$ of $\GL_\beta$, and $\SI(Q,\beta)=\bigoplus_\sigma\SI(Q,\beta)_\sigma$ decomposes into the weight spaces $\SI(Q,\beta)_\sigma=\{f\mid g\cdot f=\chi_\sigma(g)f \text{ for all } g\in\GL_\beta\}$. The occurring weights form the cone of weights $\Sigma(Q,\beta)$ \cite{derksen2000semi}, whose study yielded a proof of the saturation and Fulton conjectures for Littlewood--Richardson coefficients. Moreover, for each weight $\sigma$ the ring $\bigoplus_{M\geq 0}\SI(Q,\beta)_{M\sigma}$ is the homogeneous coordinate ring of King's moduli space of $\sigma$-semistable representations~\cite{King}, so the weight-space dimensions are the Hilbert function of this moduli space.

A recurring theme in the study of Littlewood--Richardson coefficients is their behaviour under scaling the partitions, i.e.\ under replacing a triple
$(\lambda,\mu,\nu)$ of partitions by $(M\lambda,M\mu,M\nu)$ for a positive
integer $M$. Derksen and Weyman showed that for the triple flag quiver $T$ with its standard
dimension vector $\beta$, we have
$c^{\nu}_{\lambda,\mu}=\dim_{\mathbb{K}}\SI(T,\beta)_{\sigma}$, where the weight
$\sigma=\sigma(\lambda,\mu,\nu)$ is determined from the three partitions \cite{derksen2000semi}. Scaling the partitions multiplies the weight by $M$,
\[
    c^{M\nu}_{M\lambda,M\mu}=\dim_{\mathbb{K}}\SI(T,\beta)_{M\sigma},
\]
so each statement below concerns the function
$M\mapsto\dim_{\mathbb{K}}\SI(Q,\beta)_{M\sigma}$, which makes sense for any
quiver $Q$ without oriented cycles, any dimension vector $\beta$ and any
weight $\sigma$. This correspondence allows one to transfer structural results between Littlewood--Richardson coefficients and quiver semi-invariants~\cite{CombQuivDerkWey}.

The first result of this kind is the saturation theorem of Knutson and Tao
\cite{knutson1999honeycomb}: if $c^{M\nu}_{M\lambda,M\mu}>0$ for some
$M$, then $c^{\nu}_{\lambda,\mu}>0$. In \cite{derksen2000semi}, Derksen and Weyman proved it for arbitrary quivers, i.e.\ if $\SI(Q,\beta)_{M\sigma}\neq0$ for some $M\ge1$, then
$\SI(Q,\beta)_{\sigma}\neq0$. Fulton conjectured that if $c^{\nu}_{\lambda,\mu}=1$, then $c^{M\nu}_{M\lambda,M\mu}=1$ for all
$M\ge1$. This was proved by Knutson, Tao and Woodward
\cite{knutson2004honeycomb}, and Belkale later gave a geometric proof \cite{BELKALE2007346}, which carries over to quivers. In particular, if $\dim_{\mathbb{K}}\SI(Q,\beta)_{\sigma}=1$, then
$\dim_{\mathbb{K}}\SI(Q,\beta)_{M\sigma}=1$ for all $M\ge1$ \cite[Appendix]{CombQuivDerkWey}.

King, Tollu and Toumazet conjectured that a coefficient equal to $2$ grows
linearly. More precisely, if $c^{\nu}_{\lambda,\mu}=2$, then $c^{M\nu}_{M\lambda,M\mu}=M+1$
for all $M\ge1$ \cite[Conjecture~3.3]{soton41202}. Thius was proved by Ikenmeyer \cite{Ikenmeyer_2016}, and Sherman proved the quiver version, i.e.\ if
$\dim_{\mathbb{K}}\SI(Q,\beta)_{\sigma}=2$, then
$\dim_{\mathbb{K}}\SI(Q,\beta)_{M\sigma}=M+1$ for all $M\ge1$ \cite{SHERMAN2017487}.

In the $c^{\nu}_{\lambda,\mu}=3$ case, the expected behaviour is no longer uniform.
King, Tollu and Toumazet conjecture that if $c^{\nu}_{\lambda,\mu}=3$, then
$c^{M\nu}_{M\lambda,M\mu}=2M+1$ or
$\binom{M+2}{2}$ for all $M\ge1$ \cite[Conjecture~3.3]{soton41202}. These two alternatives are the Hilbert functions of a conic and of the projective plane. The following conjecture in the quiver setting is proposed by Chindris.

\begin{conj}[King, Tollu, Toumazet; Chindris]
\label{conj:three}
  Let $Q$ be a quiver without oriented cycles. Fix a dimension vector $\beta$, and a weight $\sigma\in\Sigma(Q,\beta)$. If $\dim_{\mathbb{K}}\SI(Q,\beta)_{\sigma}=3$, then $\dim_{\mathbb{K}}\SI(Q,\beta)_{M\sigma}=2M+1$ or
  $\binom{M+2}{2}$ for all positive integers $M$.
\end{conj}
The above conjecture is particularly significant for understanding the structure of semi-invariant rings and, hence, the associated quiver moduli spaces in the sense of King. In \cite{UFDPaqPraWeh} it was shown that the $2$-dimensional weight spaces play an important role in determining the structure of the semi-invariant ring when the maximal $\GL_{\beta}$-orbits in $\rep_{\beta}(Q)$ are of codimension~1. We believe that a similar important role will be played by $3$-dimensional weight spaces when the maximal $\GL_{\beta}$-orbits in $\rep_{\beta}(Q)$ are of codimension~2.

Denote by $\SSYT(\lam,\bd a)$ the
set of semistandard Young tableau (SSYT) of shape $\lam$ and content $\bd a$. Recall that
\[
    \Kos{\lam}{\bd a}\ \coloneqq\ \#\SSYT(\lam,\bd a),
\]
is the \emph{Kostka coefficient}. We call $\Kos{\lam}{\bd a}$ a \emph{rectangular Kostka coefficient} if the partition $\lambda$ gives a Young diagram of rectangular shape. In this article we prove the above Conjecture \ref{conj:three} for all the rectangular Kostka coefficients, which is a special case of Littlewood-Richardson coefficients. More precisely, we consider the $n$-subspace quiver
\[
   Q: \;\;\begin{tikzcd}
        1 & 2 & \substack{\dotsb\\ n\;\text{vertices}} & n-1 & n \\
        && 0 &&
        \arrow[from=1-1,to=2-3]
        \arrow[from=1-2,to=2-3]
        \arrow[from=1-4,to=2-3]
        \arrow[from=1-5,to=2-3]
    \end{tikzcd}
\]
with dimension vector $\beta=\bigl(k,\overbrace{1,1,\dotsc,1}^{n\text{-many}}\bigr)$, where $n\geq k$. By the First and Second Fundamental Theorems of invariant theory, we get that the semi-invariant ring $\SI(Q,\beta)$ is isomorphic to the homogeneous coordinate ring $\mathbb{K}[\Gr(k,n)]$ of the Grassmannian $\Gr(k,n)$ of $k$-dimensional linear subspaces of $n$-dimensional space. In this case we show that

\begin{propx}[see Proposition \ref{prop:dimformula}]
     Fix $n\geq k\ge2$ and a non-negative integer $d$. The dimension of the weight spaces of $\SI(Q,\beta)$ is given by the rectangular Kostka coefficients, i.e.,
    \[
        \dim_{\mathbb{K}}\SI(Q,\beta)_{\sigma}\ =\ \Kos{(d^k)}{\bd a}\quad\text{where}\; \bd a=(a_1,\dotsc,a_n) \;\text{is a content and}\;\sigma=(-d,a_1,\dotsc,a_n).
    \]
\end{propx}
Due to the above Proposition, we will interchangeably use $\SI^{\,\Gr(k,n)}_{(d;\bd a)}$ and $\SI(Q,\beta)_{\sigma}$. With regard to Conjecture \ref{conj:three}, the relevant range of parameters is $k\geq 2$ and $n\geq k+2$, see Remark \ref{rem:trivialcase}.

\begin{thmx}[see Theorem \ref{thm:KTT}]
    Let $k\ge2$, $n\ge k+2, d\geq 1$, and $\bd a$ be a content for $(k,n,d)$. If $\Kos{(d^k)}{\bd a}=3$, then $\Kos{((Md)^k)}{M\bd a}=\binom{M+2}{2}$ or $2M+1$, for every positive integer $M$.

    In terms of the weight spaces we get that if $\dim_{\mathbb{K}}\SI^{\,\Gr(k,n)}_{(d;\bd a)}=3$, then $\dim_{\mathbb{K}}\SI^{\,\Gr(k,n)}_{(Md;M\bd a)}=\binom{M+2}{2}$ or $2M+1$, for every positive integer $M$.
\end{thmx}
The proof of above Theorem relies on finding certain lower bounds for these rectangular Kostka coefficients, or equivalently, dimension of weight spaces $\dim_{\mathbb{K}}\SI^{\,\Gr(k,n)}_{(d;\bd a)}$.

\begin{thmx}[see Proposition \ref{prop:A} and Theorem \ref{thm:star}]
\label{thm:thmC}
    Let $k\ge2$, $n\ge k+2$, $d\ge1$, and let $\bd a$ be a nondegenerate content for $(k,n,d)$, i.e.\ $\sum_ia_i=kd$ and $1\le a_i\le d-1$ for all $i$. We obtain the following bounds on the weight spaces determined by the nondegenerate content $\bd a$
    \[
        \dim_{\mathbb{K}}\SI^{\,\Gr(k,n)}_{(d;\bd a)}\ =\ \Kos{(d^k)}{\bd a}\ \ge\
        \begin{cases}
            n-2,&(k,n)=(2,n)\\[1mm]
            3,&(k,n)=(3,5)\\[1mm]
            4,&\text{otherwise}
        \end{cases}
    \]
\end{thmx}

This, together with the $\sigma$-stable decomposition (see Section \ref{sec:section2}), reduces the proof to the study of only three Grassmannians: $\Gr(2,4)$, $\Gr(2,5)$, and $\Gr(3,5)$. The last of these is isomorphic to $\Gr(2,5)$ by duality.

\begin{remark}
    With regards to Theorem~\ref{thm:thmC}, it has a geometric meaning. If $\bd a$ is nondegenerate, then $\beta$ is $\sigma$-stable by Proposition~\ref{prop:stability}, so King's moduli space \cite{King}
    \[
        \mathcal M_\sigma:=\operatorname{Proj}\bigl(\bigoplus_{M\ge0}\SI(Q,\beta)_{M\sigma}\bigr)
    \]
    of $\sigma$-semistable representations has dimension $1-\langle\beta,\beta\rangle=k(n-k)-n+1$, see \cite[Section 3]{CarChinInv}. The Hilbert function $M\mapsto\dim_{\mathbb K}\SI(Q,\beta)_{M\sigma}$, which is a polynomial by \cite{DerkWeyPoly}, has this degree. The two alternatives in
    Conjecture~\ref{conj:three} are polynomials of degree $1$ and $2$, and
    $k(n-k)-n+1\le2$ holds precisely for $(k,n)\in\{(2,4),(2,5),(3,5)\}$. Other than these
    three Grassmannians the conjecture therefore predicts that a \emph{nondegenerate content} is
    never $3$-dimensional, and this is the theme of  Theorem~\ref{thm:thmC}. We stress
    that this does not follow from the dimension count since the homogeneous coordinate ring of $\mathcal M_\sigma$ need not
    be generated in degree one.
\end{remark}

The paper is organized as follows. In Section \ref{sec:section2}, we introduce the preliminary notions concerning quiver representations and semi-invariant rings that will be used throughout the paper. In Section \ref{sec:Section 3}, we establish a connection between the semi-invariant ring of the subspace quiver and rectangular Kostka coefficients. Section \ref{sec:section 4} is devoted to reduction methods for Kostka coefficients, along with several of their properties that play a key role in the proofs of the main results presented later in the paper. In Sections \ref{sec:section 5} and \ref{sec:section 6}, we establish certain lower bounds for the weight spaces of the semi-invariant rings associated with the Grassmannians $\Gr(k,n)$. Finally, in Section \ref{sec:section 7}, we prove the King--Tollu--Toumazet conjecture for rectangular Kostka coefficients and give examples illustrating the idea of the proof.

\section{Quiver representations and semi-invariant rings}
\label{sec:section2}

We recall some important notions on quiver representations and semi-invariant rings. We refer the readers to \cite{DerkWeyBook} for a more detailed exposition.

Throughout this article $\mathbb{K}$ denotes an algebraically closed field of characteristic zero, and $\mathbb{N}$ is the set of non-negative integers. A \emph{quiver} $Q=(Q_0,Q_1)$ is a directed (multi-)graph where $Q_0$ is the set of vertices and $Q_1$ is the set of arrows. For an arrow $a\in Q_1$, we write $ta$ for the \emph{tail} of the arrow $a$ and $ha$ for the \emph{head} of the arrow $a$. A \emph{non-trivial path} in the quiver $Q$ is a sequence $p=a_na_{n-1}\dotsb a_1$ of arrows in $Q_1$ such that $ta_{i+1}=ha_i$ for every $1\leq i\leq n-1$. The head and tail of path $p$ are defined as $hp\coloneqq ha_n$ and $tp\coloneqq ta_1$, respectively. An \emph{oriented cycle} in $Q$ is a non-trivial path $p$ such that $hp=tp$. For the rest of the article, we assume $Q$ is a quiver without any oriented cycles. 

A (finite-dimensional) \emph{representation} $V$ of $Q$ over $\mathbb{K}$ is given by $V = (\lbrace V_x\rbrace_{x\in Q_0}, \lbrace V (a)\rbrace_{a\in Q_1})$ where the $V_x$ are finite dimensional $\mathbb{K}$-vector spaces and $V (a) : V_{ta}\rightarrow V_{ha}$ is a $\mathbb{K}$-linear map for each arrow $a \in Q_1$. The \emph{dimension vector} $\beta\in\mathbb{N}^{\lvert Q_0\rvert}$ of a representation $V$, denoted by $\underline{\dim}\, V$, is defined as $\beta(x)=\dim_{\mathbb{K}}V_x$. A representation $W$ of $Q$ is called a \emph{subrepresentation} of a
representation $V$ if $W_x$ is a subspace of $V_x$ for all $x \in Q_ 0$ and $W(a) : W_{ta} \rightarrow W_{ha}$ is the restriction of $V(a) :V_{ta} \rightarrow V_{ha}$, for all $a \in Q _1$.

Fix a dimension vector $\beta\in\mathbb{N}^{\lvert Q_0\rvert}$. The affine \emph{representation space} $\rep_{\beta}(Q)$ parametrizes the $\beta$-dimensional
representations of $Q$ along with a fixed basis. Denote by $\mathrm{Mat}_{a,b}$ the set of $a\times b$ matrices with entries in $\mathbb{K}$. Hence, we have
\[
    \rep_\beta(Q)= \prod_{a\in Q_1}\mathrm{Mat}_{\beta(ha),\beta(ta)}
\]
We have a regular action of the algebraic group
\[
    \GL_{\beta} \coloneqq \prod_{x\in Q_0}\GL_{\beta(x)}
\]
on $\rep_{\beta} (Q)$ given by $g\cdot V = (g_{ha}V (a)g^{-1}_{ta} )_{a\in Q_1}$ where $V=(V (a))_{a\in Q_1} \in \rep_{\beta}(Q)$ and $g=(g_x)_{x\in Q_0} \in \GL_{\beta}$. The orbits correspond to isomorphism classes of $\beta$-dimensional representations of $Q$. The action of $\GL_{\beta}$ on $\rep_{\beta}(Q)$ induces an action on the elements of the coordinate ring $\mathbb{K}[\rep_{\beta}(Q)]$ as follows: for every $g\in\GL_{\beta},f\in\mathbb{K}[\rep_{\beta}(Q)]$ and $x\in\rep_{\beta}(Q)$, we have $g\cdot f(x)\coloneqq f(g^{-1}x)$. We define
\[
    \SI(Q,\beta)\coloneqq \mathbb{K}[\rep_{\beta}(Q)]^{\SL_{\beta}},\;\text{where}\;\SL_{\beta}\coloneqq\prod_{x\in Q_0}\SL_{\beta(x)}
\]
as the \emph{ring of semi-invariant} functions on $\rep_{\beta}(Q)$.

For every weight $\sigma\in\mathbb{Z}^{\lvert Q_0\rvert}$, we can define a multiplicative character of $\GL_\beta$ as follows:
\[
    \begin{array}{ccccc}
         \chi_\sigma & : & \GL_\beta & \longrightarrow & \mathbb{K}^*\\
         \, & \, & (g_x)_{x\in Q_0} & \longmapsto & \prod_{x\in Q_0} \det(g_x)^{\sigma(x)}
    \end{array}
\]
The semi-invariant ring $\SI(Q,\beta)$ has a decomposition into weight spaces,
\[
    \SI(Q,\beta)=\bigoplus_{\sigma\in\mathbb{Z}^{\lvert Q_0\rvert}}\SI(Q,\beta)_{\sigma}
\]
where
\[
    \SI(Q,\beta)_{\sigma}\coloneqq\lbrace f\in\mathbb{K}[\rep_{\beta}(Q)]\mid g\cdot f=\chi_{\sigma}(g)\,f,\text{ for every }g\in\GL_\beta\rbrace
\]
is called the \emph{space of semi-invariants of weight $\sigma$}. A semi-invariant in a given weight space is said to be \emph{homogeneous}. We define the cone of weights
\[
    \Sigma(Q,\beta)\coloneqq\lbrace\sigma\in\mathbb{Z}^{Q_0}\mid\SI(Q,\beta)_{\sigma}\neq0\rbrace
\]

As a consequence of the Hilbert Finiteness Theorem \cite[see Theorem 9.2.6 and Corollary 9.2.8]{DerkWeyBook}, it follows that $\SI(Q,\beta)$ is a finitely generated $\mathbb{K}$-algebra. Thus, the cone of weights $\Sigma(Q,\beta)$ is a finitely generated semigroup.

For a quiver $Q=(Q_0,Q_1)$, we define the Euler form $\langle\cdot,\cdot\rangle$ on the Grothendieck group $K_{0}(\mathbb{K}Q)$ of the path algebra $\mathbb{K}Q$ of quiver $Q$, which can be identified with $\mathbb{Z}^{\lvert Q_0\rvert}$, as
\[
    \langle\alpha,\beta\rangle\coloneqq\sum_{x\in Q_0}\alpha(x)\beta(x)-\sum_{a\in Q_1}\alpha(ta)\beta(ha),\;\text{where}\;\alpha,\beta\in\mathbb{Z}^{\lvert Q_0\rvert}.
\]
We can also define weights using dimension vectors. Let $\alpha$ be a dimension vector. We define a function
\[
    \begin{matrix}
    \sigma = \langle\alpha,\cdot\rangle : \mathbb{Z}^{\lvert Q_ 0\rvert} \rightarrow \mathbb{Z} 
    \end{matrix}
\]
as follows
\[
    \sigma(x) = \alpha(x) -\sum_{\substack{a\in Q_ 1\\ha=x}}\alpha(ta)\quad\text{for every vertex}\;x\in Q_0
\]
A dimension vector $\alpha$ is said to be a \emph{generic subdimension vector} of dimension vector $\beta$, written as $\alpha\hookrightarrow \beta$, if every representation in $\rep_\beta  (Q)$ has a subrepresentation of dimension vector $\alpha$. For a weight $\sigma=(\sigma_1,\dotsc,\sigma_\ell)$ and dimension vector $\beta=(\beta_1,\dotsc,\beta_\ell)$,  we use $\sigma(\beta)=\sum_{i=1}^{\ell}\sigma_i\beta_i$ to denote the usual dot product of the two vectors in $\mathbb{Z}^{\ell}$. We have the following description of the cone of weights due to Derksen-Weyman.

\begin{theorem}[Theorem 10.7.8 \cite{DerkWeyBook}]
\label{thm:coneofweights}
    The cone of weights $\Sigma(Q,\beta)$ is a rational polyhedral cone described as follows:
    \[
        \Sigma(Q,\beta) = \lbrace\sigma = \langle\alpha,\cdot\rangle : \sigma(\beta) = 0\; \text{and for all}\;\delta \hookrightarrow \beta, \sigma(\delta) \leq 0\rbrace
    \]
    In particular $\Sigma(Q,\beta)$ is saturated, i.e.\ if $n\sigma \in \Sigma(Q,\beta)$, for some positive integer $n$, then $\sigma \in \Sigma(Q,\beta)$.
\end{theorem}

A representation $V\in\rep_{\beta}(Q)$ is said to be a \emph{general representation} with respect to a property $\mathcal{P}$ if there is an open (dense) subset $U\subseteq\rep_{\beta}(Q)$ such that $V\in U$ and every representation in $U$ satisfies the property $\mathcal{P}$. We call a representation $V\in \rep_{\beta}(Q)$ to be a \emph{Schur representation} if $\mathrm{End}(V)\cong \mathbb{K}$ as a $\mathbb{K}$-algebra. A dimension vector $\beta$ is called a \emph{Schur root} if a general representation in $\rep_{\beta}(Q)$ is a Schur representation, i.e.\ there is an open non-empty set $U\subseteq\rep_{\beta}(Q)$ such that any representation in $U$ is a Schur representation.

Let $\beta$ be a Schur root. We call $\beta$ a \emph{real} Schur root if $\langle\beta,\beta\rangle=1$, we call $\beta$ an \emph{isotropic} Schur root if $\langle\beta,\beta\rangle=0$, and we call $\beta$ an \emph{imaginary} Schur root if $\langle\beta,\beta\rangle<0$. A fact about real Schur root that we will use often is that if $\beta$ is a real Schur root then $\dim_{\mathbb{K}}\SI(Q,\beta)_{\sigma}\leq 1$ for every weight $\sigma$, see \cite[Lemma 9.4.2]{DerkWeyBook} or \cite{CKW}.

Due to \cite{King}, we have the following notion of $\sigma$-(semi)stability with respect to a weight $\sigma\in\mathbb{Z}^{\lvert Q_0\rvert}$. Fix a weight $\sigma\in\mathbb{Z}^{\lvert Q_0\rvert}$. A representation $V\in\rep_\beta(Q)$ is called \emph{$\sigma$-semistable} (resp.\ \emph{$\sigma$-stable}) if $\sigma(\beta)=0$ and $\sigma(\underline{\dim} W)\le 0$ (resp.\ $\sigma(\underline{\dim} W)<0$) for every subrepresentation $W\subseteq V$ (resp.\ every nonzero proper subrepresentation $0\neq W\subsetneq V$).

From \cite[Section 3.2]{CombQuivDerkWey} we know that a dimension vector $\beta$ is called $\sigma$-(semi)stable if a general representation in $\rep_{\beta}(Q)$ is $\sigma$-(semi)stable. Suppose $\beta$ is a $\sigma$-semistable dimension vector. A decomposition
\[
    \beta=\beta_1+\beta_2+\dotsc+\beta_t
\]
is called the \emph{$\sigma$-stable} decomposition of $\beta$ if a general representation in $\rep_{\beta}(Q)$ has a Jordan-H\"older filtration with factors of dimension vector $\beta_1,\beta_2,\dotsc,\beta_t$ in some order. We write the $\sigma$-stable decomposition of $\beta$ as
\[
    \beta=\beta_1\dotplus\beta_2\dotplus\dotsc \dotplus\beta_t
\]

\section{Grassmannians and rectangular Kostka coefficients}
\label{sec:Section 3}

Consider the $n$-subspace quiver
\[
    \begin{tikzcd}
        1 & 2 & \substack{\dotsb\\ n\;\text{vertices}} & n-1 & n \\
        && 0 &&
        \arrow[from=1-1,to=2-3]
        \arrow[from=1-2,to=2-3]
        \arrow[from=1-4,to=2-3]
        \arrow[from=1-5,to=2-3]
    \end{tikzcd}
\]
with dimension vector $\beta=\bigl(k,\overbrace{1,1,\dotsc,1}^{n\text{-many}}\bigr)$, where $n\geq k$. As a consequence of the First and Second Fundamental Theorems of invariant theory (also see \cite[Chapter 2, Section 2.2 \& 2.3]{IgorInvTh}), we get that the semi-invariant ring
\[
    \SI(Q,\beta)\cong\mathbb{K}[\Gr(k,n)]
\]
the homogeneous coordinate ring of the Grassmannian $\Gr(k,n)$ of $k$-dimensional linear subspaces of $n$-dimensional space.

For $I\subseteq\{1,\dotsc,n\}$, we will write $\mathbf{1}_{I}\in\mathbb{Z}^{n}$ for its
indicator vector so that $(\mathbf{1}_{I})_i=1$ if $i\in I$, and $0$ otherwise. Hence, we have $\mathbf{1}_{\varnothing}=0$, $\mathbf{1}_{\{i\}}=e_i$ and $\mathbf{1}_{\{1,\dotsc,n\}}=\mathbf{1}$.

\begin{prop}
\label{prop:stability}
    Let $k\ge 2$.
    {\renewcommand{\labelenumi}{(\alph{enumi})}%
     \renewcommand{\labelenumii}{(\roman{enumii})}%
    \begin{enumerate}
      \item Any weight $\sigma\in\Sigma(Q,\beta)$ is of the form
            $\sigma=(-d,a_1,a_2,\dotsc,a_n)$ where $d,a_i\in\mathbb{N}$ with $\sum_ia_i=kd$ and $a_i\leq d$
      \item For such a weight $\sigma\in\Sigma(Q,\beta)$, the dimension vector $\beta$ is
            \begin{enumerate}
              \item $\sigma$-semistable if and only if $0\le a_i\le d$ for all $i$
              \item $\sigma$-stable if and only if $1\le a_i\le d-1$ for all $i$ (assume $d\geq 2$).
            \end{enumerate}
    \end{enumerate}}
    
\end{prop}

\begin{proof}
    The possible generic subdimension vectors are
    \[
        \gamma=(r,\mathbf 1_I)\ \text{ with }|I|\le r\le k-1\qquad\text{and}\qquad
        \gamma=(k,\mathbf 1_I)\ \text{ with }I\subseteq\{1,\dots,n\}\ \text{arbitrary},
    \]
    along with $\gamma=0$ and $\gamma=\beta$.
    
    (a) Consider the following three subdimension vectors
      \[
        (k,0,\dotsc,0),\quad (k,\mathbf{1}-e_i), \quad (1,e_i)
    \]
    
    Let $\sigma=(\sigma_0,\sigma_1,\dotsc,\sigma_n)\in\Sigma(Q,\beta)$. Using the description of the cone of weights $\Sigma(Q,\beta)$ as in Theorem \ref{thm:coneofweights}, the equality $\sigma(\beta)=0$ and the above subdimension vectors give the following
    \[
      k\sigma_0+\sum_{j}\sigma_j=0,\qquad
      k\sigma_0\le 0,\qquad
      k\sigma_0+\sum_{j\neq i}\sigma_j\le 0,\qquad
      \sigma_0+\sigma_i\le 0 .
    \]
    The first inequality gives $\sigma_0\leq 0$. We will write $\sigma_0=-d$ with $d\ge 0$. If we subtract the second inequality from the equality we get $\sigma_i\ge 0$. Let us write $a_i=\sigma_i$. The third inequality says $a_i\le d$, and the equality gives $\sum_i a_i=kd$. Conversely, it is easy to check that any weight $\sigma$ of this form lies in the cone of weights $\Sigma(Q,\beta)$.

    (b) We will only prove part $(ii)$. The proof of part $(i)$ is similar. 
    
    $(\Longrightarrow)$ Taking $\gamma=(1, e_i)$ gives $\sigma(\gamma)=a_i-d<0$. Hence, $a_i\le d-1$. If we take $\gamma=(k, \mathbf 1_{\{1,\dots,n\}\smallsetminus\{i\}})$, which is a
    proper subrepresentation, then we get $\sigma(\gamma)=-kd+(kd-a_i)=-a_i<0$. Hence, $a_i\ge1$.
    
    $(\Longleftarrow)$ Assume $1\leq a_i\leq d-1$. For $\gamma=(r,\mathbf 1_I)$ with
    $1\le r\le k-1$ and $|I|\le r$,
    \[
    \sigma(\gamma)=-rd+\sum_{i\in I}a_i\ \le\ -rd+|I|(d-1)\ \le\ -rd+r(d-1)=-r<0 .
    \]
    For $\gamma=(k,\mathbf 1_I)$ with $I\subsetneq\lbrace1,2,\dotsc,n\rbrace$, we get that $\sigma(\gamma)=-kd+\sum_{i\in I}a_i=-\sum_{i\notin I}a_i<0$ since some $a_i\ge1$.
\end{proof}

This leads to the following definition:

\begin{defn}\label{def:content}
    Fix $n\geq k\ge2$ and a non-negative integer $d$. A \emph{content} for $(k,n,d)$ is a sequence $\bd a=(a_1,\dots,a_n)$ of non-negative integers such that $a_i\leq d$ and $\sum_{i=1}^{n}a_i=kd$. Additionally, if the content $\bd a$ satisfies $1\leq a_i\leq d-1$ for all $1\leq i\leq n$, then it is said to be \emph{nondegenerate} (assume $d\geq 2$).
\end{defn}

We recall the definition of Kostka coefficients.

\begin{defn}\label{def:ssyt}
    Let $\lam$ be a partition and $n\ge1$. A \emph{semistandard Young tableau} (SSYT) $T$ of shape $\lam$ in the letters $\{1,\dots,n\}$ is a filling of the Young diagram corresponding to $\lam$ such that
    \begin{enumerate}[label=(\roman*),nosep]
        \item along each row from left to right the entries are weakly increasing, and
        \item along each column the entries are strictly increasing.
    \end{enumerate}
    The \emph{content} of $T$ is a sequence of non-negative integers $(a_1,\dots,a_n)$ where $a_i$ is
    the number of cells of $T$ containing the letter $i$. We will write $\SSYT(\lam,\bd a)$ for the
    set of semistandard Young tableau (SSYT) of shape $\lam$ and content $\bd a$. Hence,
    \[
        \Kos{\lam}{\bd a}\ \coloneqq\ \#\SSYT(\lam,\bd a),
    \]
    is the \emph{Kostka coefficient}. If the partition $\lambda$ is a Young diagram of rectangular shape, then we will call $\Kos{\lam}{\bd a}$ a \emph{rectangular Kostka coefficient}.
\end{defn}
As a consequence of \cite[Theorem 7.10.2]{StanVol2} we know that $\Kos{\lam}{\bd a}$ does not change if the entries of $\bd a$ are permuted, so, when needed, we will sort the entries of $\bd a$.

Let $W\cong\mathbb{K}^m$ be a $m$-dimensional vector space. Recall that for a partition $\rho$ with at most $m$ parts, the \emph{Schur module} $\mathbb{S}_{\rho}W$ is the irreducible representation of the general linear group $\GL(W)\cong\GL_{m}$ with highest weight $\rho$.

\begin{prop}
\label{prop:dimformula}
     Fix $n\geq k\ge2$ and a non-negative integer $d$. The dimension of the weight spaces of $\SI(Q,\beta)$ is given by the rectangular Kostka coefficients, i.e.,
    \[
        \dim_{\mathbb{K}}\SI(Q,\beta)_{\sigma}\ =\ \Kos{(d^k)}{\bd a}\quad\text{where } \bd a \;\text{is a content for}\;(k,n,d).
    \]
\end{prop}

\begin{proof}
    For calculating $\dim_{\mathbb{K}}\SI(Q,\beta)_{\sigma}$ we will use the description from \cite[Section 10.4]{DerkWeyBook}. Let $V$ be a representation of $Q$ with $\underline{\dim}V=\beta$. Now $\dim_{\mathbb{K}}\SI(Q,\beta)_{\sigma}$ is the multiplicity of
    \begin{align*}
      &\mathbb{S}_{((-d)^k)}V(0)\otimes\mathbb{S}_{((a_1)^1)}V(1)\otimes\mathbb{S}_{((a_2)^1)}V(2)
       \otimes\dotsb\otimes\mathbb{S}_{((a_n)^1)}V(n)\\
      &\qquad\cong\mathbb{S}_{((d)^k)}V(0)^{*}\otimes\mathbb{S}_{((a_1)^1)}V(1)\otimes\mathbb{S}_{((a_2)^1)}V(2)
       \otimes\dotsb\otimes\mathbb{S}_{((a_n)^1)}V(n)
    \end{align*}
    in
    \begin{align*}
      \mathbb{K}[\rep_{\beta}(Q)]
      &= \bigoplus_{\substack{\lambda^{(1)},\dotsc,\lambda^{(n)}\\ \text{partitions}}}\;
         \bigotimes_{i=1}^{n}
         \Bigl(\mathbb{S}_{\lambda^{(i)}}V(i)\otimes\mathbb{S}_{\lambda^{(i)}}V(0)^{*}\Bigr)\\
      &\cong \bigoplus_{\substack{\lambda^{(1)},\,\dotsc\,,\lambda^{(n)}\\ \text{partitions}}}
         \Bigl(\,\bigotimes_{i=1}^{n}\mathbb{S}_{\lambda^{(i)}}V(i)\Bigr)
         \otimes
         \Bigl(\,\bigotimes_{i=1}^{n}\mathbb{S}_{\lambda^{(i)}}V(0)^{*}\Bigr)
    \end{align*}
    Therefore, we get that $\lambda^{(i)}=(a_i)^1$ for every $1\leq i\leq n$. Hence, we are looking for the multiplicity of
    \[
        \mathbb{S}_{((d)^k)}V(0)^{*}
    \]
    in
    \[
        \bigotimes_{i=1}^{n}\mathbb{S}_{(a_{i})}V(0)^{*}
    \]
    Hence, we conclude that
    \[
        \dim_{\mathbb{K}}\SI(Q,\beta)_{\sigma}\ =\ \Kos{(d^k)}{\bd a}
    \]
\end{proof}

\begin{remark}
    We have two consequences of Proposition~\ref{prop:dimformula} that we will use throughout.
    \begin{enumerate}
        \item The term \emph{content} has been used in two ways: for weights of the subspace quiver, see Definition~\ref{def:content}, and for tableaux, see Definition~\ref{def:ssyt}. We note that both the notions agree. Indeed, by Proposition~\ref{prop:stability} and Proposition~\ref{prop:dimformula}, a sequence $\bd a$ is a content for $(k,n,d)$ if and only if $\Kos{(d^k)}{\bd a}\neq0$, i.e., if and only if $\bd a$ is the content of some semistandard tableau of shape $(d^k)$.
        \item We will often write $\SI^{\,\Gr(k,n)}_{(d;\bd a)}$ for the weight space $\SI(Q,\beta)_{\sigma}$ with $\sigma=(-d,a_1,\dots,a_n)$ and use them interchangeably. We drop the superscript when $(k,n)$ is clear from the context.
    \end{enumerate}
\end{remark}

\section{Some properties of Kostka coefficients via semi-invariants}
\label{sec:section 4}

Below we record some reduction methods for Kostka coefficients, along with their interpretation in terms of the semi-invariant ring. We also recall some well-known properties of the Kostka coefficients.

\begin{lem}\label{lem:T0}
    Let $\sigma=(-d,a_1,a_2,\dotsc,a_n)$ be a weight. If $a_i=0$ for some $i$, then $\SI^{\Gr(k,n)}_{(d;\bd a)}\cong\SI^{\Gr(k,n-1)}_{(d;\bd a')}$, where $\bd a'$ is obtained by deleting the $i$-th entry of $\bd a$. Combinatorially, let $\bd a$ be a content for $(k,n,d)$ with $a_i=0$, and let $\bd a'$ be the content obtained by deleting the $i$-th entry. Then $\Kos{(d^k)}{\bd a}=\Kos{(d^k)}{\bd a'}$.
\end{lem}

\begin{proof}
    A tableau of content $\bd a$, with $a_i=0$, never uses the letter $i$. So, it is a tableau in the alphabet $\{1,\dots,n\}\smallsetminus\{i\}$. We relabel that alphabet order-preservingly as $\{1,\dots,n-1\}$ to get the desired result.
\end{proof}

Using the duality of the Grassmannians $\Gr(k,n)\cong\Gr(n-k,n)$ we get

\begin{lem}\label{lem:T2}
    Let $\bd a$ be a content for $(k,n,d)$. Define $\bd a^\vee:=(d-a_1,\dots,d-a_n)$. For every integer $m\ge1$, we have
    \[
        \dim_{\mathbb{K}}\SI_{(md;m\bd a)}^{\,\Gr(k,n)}\ =\ \dim_{\mathbb{K}}\SI_{(md;m\bd a^\vee)}^{\,\Gr(n-k,n)}
    \]
    In terms of Kostka coefficients we get
    \[
        \Kos{(d^k)}{\bd a}\ =\ \Kos{(d^{\,n-k})}{\bd a^{\vee}}
    \]
    Moreover, $\bd a$ is nondegenerate if and only if $\bd a^\vee$ is.
\end{lem}

For a given content $\bd a$ of $(k,n,d)$, if one of the entries $a_i=d$, then we can use a projection method to reduce to a ``smaller'' Grassmannian.

\begin{lem}\label{lem:T3}
    Suppose $\bd a$ is a content for $(k,n,d)$ with $a_i=d$ for some $i$. Let $\bd a'$ be
    obtained by deleting the $i$-th entry. For every integer $m\ge1$, we have
    \[
        \dim_{\mathbb{K}}\SI_{(md;m\bd a)}^{\,\Gr(k,n)}\ =\ \dim_{\mathbb{K}}\SI_{(md;m\bd a')}^{\,\Gr(k-1,n-1)}
    \]
    Combinatorially, we get
    \[
        \Kos{(d^k)}{\bd a}\ =\ \Kos{(d^{\,k-1})}{\bd a'}
    \]
\end{lem}

\begin{proof}
    Let $\sigma=(-d,a_1,\dotsc,a_n)$ be the weight corresponding to the content $\bd a$. Consider the dimension vector $\gamma=(1,e_i)$, i.e.\ $1$ at central vertex and $1$ at the $i$-th vertex. Since $a_i=d$, the content $\bd a$ is not nondegenerate, i.e.\ $\beta$ is not $\sigma$-stable. Notice that $\gamma$ is a real Schur root and it is $\sigma$-stable. Moreover, the $\sigma$-stable decomposition of $\beta$ will have $\gamma$ as one of its summands, which will occur with multiplicity one since $\beta=(k,1,1,\dotsc,1)$. Therefore, as a consequence of \cite[Theorem 3.20]{CombQuivDerkWey} we get
    \[
        \SI(Q,\beta)_{m\sigma}\cong\SI(Q,\gamma)_{m\sigma}\otimes\SI(Q,\beta-\gamma)_{m\sigma}
    \]
    We have $\SI(Q,\beta-\gamma)\cong\mathbb{K}[\Gr(k-1,n-1)]$. Hence, the result follows since $\gamma$ is a real Schur root and so $\dim_{\mathbb{K}}\SI(Q,\gamma)_{m\sigma}\leq 1$.
\end{proof}

We further investigate some properties of the Kostka coefficients. We recall some well-known notions from \cite[Chapter I, Section 1]{Macdonald}.

\begin{defn}
\label{def:hstrip}
    Let $\mu\subseteq\lam$ be partitions. The skew diagram $\lam/\mu$ is a \emph{horizontal strip} if it has at most one cell in each column.
\end{defn}

\begin{lem}
\label{lem:interlace}
    Let $\mu\subseteq\lam$ be partitions. The skew diagram $\lam/\mu$ is a horizontal strip if and only if
    \begin{equation}
    \label{eq:interlace}
        \lam_1\ \ge\ \mu_1\ \ge\ \lam_2\ \ge\ \mu_2\ \ge\ \lam_3\ \ge\ \mu_3\ \ge\ \cdots
    \end{equation}
\end{lem}

The next observations are worth isolating.

\begin{lem}\label{lem:stripbottom}
    Let $\lam/\mu$ be a horizontal strip and let $(i,j)\in\lam/\mu$. Then $(i+1,j)\notin\lam$. 
\end{lem}

\begin{proof}
Suppose $(i+1,j)\in\lam$, i.e.\ $j\le\lam_{i+1}$. Since $\lam/\mu$ has at most one cell per
column and $(i,j)\in\lam/\mu$, the cell $(i+1,j)$ must lie in $\mu$, i.e.\ $j\le\mu_{i+1}$.
But $\mu_{i+1}\le\mu_i<j$, since $(i,j)\notin\mu$. Contradiction.
\end{proof}

\begin{remark}
    The above Lemma \ref{lem:stripbottom} says that a cell of a horizontal strip has no cell of $\lam$ directly below it, and hence it sits at the bottom of its column in $\lam$.
\end{remark}

\begin{lem}
\label{lem:striporder}
    Let $\lam/\mu$ be a horizontal strip. If $(i,j)$ and $(i',j')$ are cells of $\lam/\mu$ with
    $i<i'$, then $j'<j$. In particular, cells in the lower rows of a horizontal strip occupy strictly smaller columns.
\end{lem}

\begin{proof}
The cells of $\lam/\mu$ in row $r$ occupy the columns $\mu_r+1,\dots,\lam_r$. So
$j'\le \lam_{i'}\le\lam_{i+1}\le\mu_i<j$, using \eqref{eq:interlace} in the middle.
\end{proof}

We recall another well-known property of elements in $\SSYT(\lam,\bd a)$.

\begin{lem}
\label{lem:shapeofletters}
    Let $T\in\SSYT(\lam,\bd a)$. For $0\le j\le n$, if we define
    \[
        \lam^{(j)}:=\{\text{cells of }\lam\text{ whose entry is }\le j\},
    \]
    then each $\lam^{(j)}$ is the Young diagram of a partition. Additionally, we have
    \[
        \varnothing=\lam^{(0)}\subseteq\lam^{(1)}\subseteq\cdots\subseteq\lam^{(n)}=\lam ,
    \]
    with $\lam^{(j)}/\lam^{(j-1)}$ a horizontal strip of size $a_j$ for every $j$.
\end{lem}

The result of Lemma \ref{lem:polarization} below follows from the monotonicity of Kostka coefficients, see \cite[Section II, Theorem]{KostMono}. However, we include a proof here since the idea is useful for Proposition \ref{prop:T1image} below.

Let $\bd a=(a_1,\dots,a_m,p,q)$ be a content in the letters $1,\dots,m+2$. Its \emph{merge} (in the last two letters) is the content
\[
    \bd a^{\,\mathrm{mrg}}\coloneqq(a_1,\dots,a_m,\,p+q)
\]
in the letters $1,\dots,m+1$. We will write $\ell=m+1$ for the largest letter for the content $\bd a^{\,\mathrm{mrg}}$, and for the content $\bd a$ the two largest letters are $\ell=m+1$ and $\ell+1=m+2$.

\begin{lem}
\label{lem:polarization}
    With the notation as above, for every partition $\lam$, and all $p,q\ge0$, we have
    \[
        \Kos{\lam}{(a_1,\dots,a_m,\,p,\,q)}\ \ge\ \Kos{\lam}{(a_1,\dots,a_m,\,p+q)} .
    \]
\end{lem}

\begin{proof}
    Let $T'\in\SSYT(\lam,\bd a^{\,\mathrm{mrg}})$. By Lemma~\ref{lem:shapeofletters} the cells
    of $T'$ containing $\ell$ form a horizontal strip. We denote this horizontal strip by $S:=\lam/\mu$, where $\mu$ is the shape
    occupied by the letters $\le m$. Note that it has $p+q$ cells. Since $S$ is a horizontal strip, its cells lie in distinct columns. We list them in the increasing order of column index
    \[
        s_1,\ s_2,\ \dots,\ s_{p+q}.
    \]
    By Lemma~\ref{lem:striporder}, cells in lower rows come earlier since they occupy strictly smaller columns.
    
    Define $\Phi(T')$ to be the tableau obtained from $T'$ by relabelling
    \[
        s_1,\dots,s_p\ \longmapsto\ \ell,\qquad s_{p+1},\dots,s_{p+q}\ \longmapsto\ \ell+1,
    \]
    and leaving all other cells unchanged. We note that its content is $(a_1,\dots,a_m,p,q)$ by construction.
    
    We check semistandardness condition for rows. Only the entries in the cells of $S$ have changed, and their labels are $\ell$ or $\ell+1$, both of which are larger than every letter $\le m$. Hence, a row will fail the semistandardness condition only if
    some $\ell+1$ occurs strictly to the left of an $\ell$ within the same row. We note that the cells of $S$ occupy a set of consecutive columns within a row, namely $\mu_r+1,\dots,\lam_r$ in row $r$. Therefore, they are consecutive, and the relabeling assigns them $\ell$'s followed by $\ell+1$'s from left to right. Hence, the rows are weakly increasing.
    
    We check semistandardness condition for columns. By Lemma~\ref{lem:stripbottom}, each cell of $S$ has no cell of $\lam$ directly below it. Therefore, a cell of $S$ is the bottom cell of its column. Thus, for columns to be strictly increasing, we need to check that the entry in the relabeled cell exceeds the entry in the cell directly above it, which lies in $\mu$. Hence, it has an entry $\le m<\ell$. Therefore, all columns are strictly increasing.
    
    Finally $\Phi$ is injective since relabelling every $\ell+1$ back to $\ell$ recovers $T'$.
\end{proof}

In terms of semi-invariants, the above Lemma \ref{lem:polarization} can be interpreted as follows.

\begin{cor}
    Let $\bd a=(a_1,\dots,a_{n-2},p,q)$ and $\bd a^{\,\mathrm{mrg}}=(a_1,\dots,a_{n-2},p+q)$ be defined as above, with $p,q\ge0$. If $\lam$ is a rectangular partition, then there is an injective linear map
    \[
        \mathcal F:\ \SI^{\,\Gr(k,n-1)}_{(d;\bd a^{\,\mathrm{mrg}})}\ \longrightarrow\
        \SI^{\,\Gr(k,n)}_{(d;\bd a)}
    \]
    In particular, $\dim_{\mathbb{K}}\SI^{\,\Gr(k,n)}_{(d;\bd a)}\ \ge\ \dim_{\mathbb{K}}\SI^{\,\Gr(k,n-1)}_{(d;\bd a^{\,\mathrm{mrg}})}$.
\end{cor}

If the partition $\lam$ is rectangular, then the map $\Phi$ is a bijection and we have a complete description of its image.

\begin{prop}
\label{prop:T1image}
    Let $\lam=(d^k)$ be a rectangle, and
    $\bd a=(a_1,\dots,a_{\ell-1},p,q)$ be a content with $\ell$ and $\ell+1$ as the two largest letters. In such a case, the injection $\Phi$ of Lemma~\ref{lem:polarization} is a bijection. More precisely,
    \[
        \Phi:\ \SSYT\bigl((d^k),\bd a^{\,\mathrm{mrg}}\bigr)\ \xrightarrow{\ \sim\ }\
        \bigl\{T\in\SSYT\bigl((d^k),\bd a\bigr)\ :\ \text{no column of $T$ contains both $\ell$ and $\ell+1$}\bigr\}.
    \]
    In particular,
    \begin{equation}
    \label{eq:strict}
        \Kos{(d^k)}{\bd a}\ =\ \Kos{(d^k)}{\bd a^{\,\mathrm{mrg}}}\ +\
        \#\bigl\{T\in\SSYT((d^k),\bd a):\ \text{some column contains both $\ell$ and $\ell+1$}\bigr\}.
    \end{equation}
\end{prop}

\begin{proof}
    We refer to the same construction as in the proof of Lemma \ref{lem:polarization}.
    
    $(\subseteq)$ If $T=\Phi(T')$ then the cells of $T$ carrying $\ell$ or $\ell+1$ are exactly the
    cells of the horizontal strip $S$ of $T'$, and hence lie in distinct columns. Therefore, no column contains both $\ell$ and $\ell+1$.
    
    $(\supseteq)$ Let $T\in\SSYT((d^k),\bd a)$ have no column containing both $\ell$ and $\ell+1$.
    Suppose $\mu$ is the diagram occupied by the letters $\le \ell-1$, which we know is a partition, by Lemma~\ref{lem:shapeofletters}. The skew diagram $\lam/\mu$ consists of the cells filled with letters $\ell$ or $\ell+1$. Since columns are strictly increasing, each column contains at most one $\ell$ and at most one $\ell+1$ and, by hypothesis, not both. Therefore, $\lam/\mu$ has at most one cell per column, and hence is a horizontal strip. Since $\lam=(d^k)$, the interlacing
    property in Lemma \ref{lem:interlace} forces $\mu_i=d$ for $i\le k-1$. Hence, $\lam/\mu$ lies in row $k$ and consists of the last
    $p+q$ cells of that row. Since row $k$ is weakly increasing, the letters $\ell$'s precede the letters $\ell+1$'s along this row. Since there are $p$ of the former and $q$ of the latter, $T$ assigns $\ell$ to the first $p$ and $\ell+1$ to the last $q$ cells of the horizontal strip in column order. Relabelling every $\ell+1$ to $\ell$ therefore gives a semistandard tableau $T'$ of merged content $\bd a^{\,\mathrm{mrg}}$ with $\Phi(T')=T$.
\end{proof}

We obtain certain bounds on the entries $a_i$ which will be useful for the proofs in the later sections.

\begin{lem}
\label{lem:pairbound}
    Let $a_1\ge a_2\ge\cdots\ge a_n\ge1$ with $\sum_i a_i=A$ and $n\ge2$. We have
    \[
        a_{n-1}+a_n\ \le\ \frac{2A}{n},
    \]
    with equality only if $a_1=a_2=\cdots=a_n$.
\end{lem}

\begin{proof}
    Since the sequence is sorted, $a_{n-1}$ and $a_n$ are the two smallest entries. Therefore,  every pair $\{i,j\}$ with $i<j$ satisfies $a_i+a_j\ge a_{n-1}+a_n$. Summing over all such
    $\binom {n}{2}$ pairs, and noting that each index occurs in exactly $n-1$ such pairs, we get
    \[
        (n-1)A\ =\ \sum_{i<j}(a_i+a_j)\ \ge\ \binom n2\,(a_{n-1}+a_n)\ =\ \frac{n(n-1)}{2}(a_{n-1}+a_n)
    \]
    This proves the claim. 
    
    For the equality case, note that every pair sum has to be equal to $a_{n-1}+a_n$. In
    particular, if $n\ge3$, and $j\ne l$ are both distinct from some $i$, then we get $a_i+a_j=a_i+a_l$. Therefore, $a_j=a_l$. Letting $j$ vary, we get that all entries are equal. For $n=2$ the statement is vacuously true since $a_1+a_2=A$.
\end{proof}

\begin{cor}
\label{cor:pairbound}
    Let $\bd a$ be a content for $(k,n,d)$, sorted decreasingly, with all $a_i\ge1$. We have
    \[
        a_{n-1}+a_n\ \le\ \frac{2kd}{n} .
    \]
    In particular, if $n\ge 2k+1$, then $a_{n-1}+a_n<d$. Additionally, if $n=2k$, then $a_{n-1}+a_n\le d$, with $d$ even, and equality occurs only when all $a_i$ are equal, necessarily to $d/2$.
\end{cor}

\begin{proof}
    We will use the above Lemma~\ref{lem:pairbound} with $A=kd$. For $n\ge2k+1$ we get
    $a_{n-1}+a_n\le 2kd/n\le 2kd/(2k+1)<d$. For $n=2k$ we get $a_{n-1}+a_n\le d$ with the
    stated equality case, and if all $a_i=a$, then $2ka=kd$, so  $d=2a$, which is even. Hence, $a=d/2$.
\end{proof}

\section{A lower bound in Grassmannian \texorpdfstring{$\Gr(2,n)$}{Gr(2,n)} and \texorpdfstring{$\Gr(3,6)$}{Gr(3,6)}}
\label{sec:section 5}

In this section we establish a certain lower bound for the dimensions of the weight spaces of $\SI^{\,\Gr(2,n)}_{(d;\bd a)}$ and $\SI^{\,\Gr(3,6)}_{(d;\bd a)}$ for nondegenerate content $\bd a$, as defined in Definition \ref{def:content}.

\subsection{Grassmannian \texorpdfstring{$\Gr(2,n)$}{Gr(2,n)}}

Throughout this subsection $k=2$. For a content $\bd a=(a_1,\dots,a_n)$ of $(2,n,d)$ we will write
\[
    d=\tfrac12\textstyle\sum_i a_i
\]
For simplicity, we write $F(\bd a)\coloneqq\Kos{(d,d)}{\bd a}$. Hence, $F(\bd a)=0$ if $\sum_i a_i$ is odd. 

We will write $V_m\coloneqq\mathrm{Sym}^m\, V$ for the unique $(m+1)$-dimensional irreducible representation of $\mathrm{SL}_{2}$, which is the $m$-th symmetric power of the standard representation $V\cong \mathbb{K}^2$. Recall the Clebsch-Gordan decomposition for $\mathrm{SL}_{2}$ representations, see \cite[Section 11.2]{FultonHarrisRepTh} or \cite[Proposition V.5.1]{Kassel1995}.

\begin{lem}\label{lem:CG}
    For $a,b\ge0$,
    \[
        V_a\otimes V_b\ \cong\ \bigoplus_{t=0}^{\min(a,b)} V_{a+b-2t}
        \qquad\text{as }\mathrm{SL}_2\text{-representations}
    \]
\end{lem}

\begin{remark}
    Note that $\dim(V_a\otimes V_{b})^{\SL_2}=\delta_{a,b}$. Indeed, because $V_a$ is self-dual and irreducible, so $\dim(V_a\otimes V_{b})^{\SL_2}=\dim\mathrm{Hom}_{\SL_2}(V_a,V_{b})
    =\delta_{a,b}$.
\end{remark}

\begin{prop}\label{prop:fusion}
    For $n\ge3$,
    \begin{equation}
    \label{eq:fusrecur}
        F(a_1,\dots,a_n)\ =\ \sum_{t=0}^{\min(a_{n-1},a_n)} F\bigl(a_1,\dots,a_{n-2},\,j_t\bigr),
        \qquad j_t:=a_{n-1}+a_n-2t .
    \end{equation}
    The $t$-th summand is a content in $n-1$ letters whose half-sum is $d-t$.
\end{prop}

\begin{proof}
    Apply Lemma~\ref{lem:CG} to the last two tensor factors and use that taking
    $\SL_2$-invariants is additive over direct sums. The half-sum of $(a_1,\dots,a_{n-2},j_t)$ is $\tfrac12(2d-2t)=d-t$.
\end{proof}

\begin{lem}\label{lem:star}
If $\bd a$ is a content for $(2,n,d)$ with $a_i=d$ for some $i$, then $F(\bd a)=1$.
\end{lem}

\begin{proof}
    By Lemma~\ref{lem:T3}, $F(\bd a)=\Kos{(d^1)}{\bd a'}=\Kos{(d)}{\bd a'}$ where $\bd a'$
    deletes the $i$-th entry. A tableau with a single row of length $d$ is determined by
    its content by writing the letters in weakly increasing order. Hence, $F(\bd a)=1$.
\end{proof}

\begin{prop}\label{prop:A}
    Let $n\ge3$. If $\bd a$ is a nondegenerate content for $(2,n,d)$, then
    \[
        F(\bd a)\ \ge\ n-2 .
    \]
\end{prop}

\begin{proof}
    As a consequence of \cite[Theorem 7.10.2]{StanVol2}, we assume $a_1\ge a_2\ge\cdots\ge a_n\ge1$. Hence, $a_1\le d-1$. We use induction on $n$.
    
    Assume $n=3$. We claim $F(\bd a)=1$. By
    Proposition~\ref{prop:fusion} and Lemma~\ref{lem:CG},
    \[
        F(a_1,a_2,a_3)=\sum_{t}\dim\bigl(V_{j_t}\otimes V_{a_1}\bigr)^{SL_2}
    \]
    is $1$ if $a_1$ occurs among the $j_t$ which is equivalent to $|a_2-a_3|\le a_1\le a_2+a_3$ with
    $a_1\equiv a_2+a_3\bmod 2$; and $0$ otherwise. Now $a_1+a_2+a_3=2d$ is even, giving the
    parity condition. The inequality $a_1\le a_2+a_3=2d-a_1$ is equivalent to $a_1\le d$, which is true. Using
    $a_2\ge a_3$, we get that $a_2-a_3\le a_1$ is equivalent to $a_2\le a_1+a_3=2d-a_2$, i.e.\ $a_2\le d$, which is
    true. Hence, $F(\bd a)=1=n-2$.
    
    Now we assume $n\geq 4$. Since $k=2$, by Corollary~\ref{cor:pairbound} we get,
    \begin{equation}\label{eq:pairA}
        a_{n-1}+a_n\ \le\ \frac{4d}{n}\ \le\ d ,
    \end{equation}
    and the second inequality is an equality only if $n=4$. Additionally, the first is an equality
    only if all $a_i$ are equal. In other words,
    \begin{equation}\label{eq:pairA2}
        a_{n-1}+a_n\le d,\quad\text{and}\quad a_{n-1}+a_n\le d-1
    \ \text{unless } n=4 \text{ and } a_1=a_2=a_3=a_4 .
    \end{equation}
    We consider the following two cases according to the size of $a_1+a_n$.
    
    \noindent\textbf{Case A: $a_1+a_n\le d-1$.} Consider the summand for $t=0$ in \eqref{eq:fusrecur}, whose content is
    $\bd b_0=(a_1,\dots,a_{n-2},\,a_{n-1}+a_n)$ with half-sum $d$. We check the nondegeneracy of the content $\bd b_0$. All entries of $\bd b_0$ are $\ge1$. The entries of $\bd b_0$ are also $\le d-1$, since $a_i\le a_1\le d-1$ for $i\le n-2$ and
    $a_{n-1}+a_n\le a_1+a_n\le d-1$. So $\bd b_0$ is a nondegenerate content for
    $(2,n-1,d)$. Hence, by induction hypothesis, $F(\bd b_0)\ge(n-1)-2=n-3$.
    
    Now, assume $a_{n-1}\ge2$. Consider also the summand for $t=1$ whose content is $\bd b_1=(a_1,\dots,a_{n-2},\,j_1)$ with $j_1=a_{n-1}+a_n-2$, and half-sum $d-1$. We check the nondegeneracy of the content $\bd b_1$. We have $j_1\ge a_{n-1}+1-2\ge1$. Each $a_i\le a_1\le d-1-a_n\le d-2=(d-1)-1$; and $j_1\le a_1+a_n-2\le d-3\le (d-1)-1$. So $\bd b_1$ is a nondegenerate content for $(2,n-1,d-1)$. By induction hypothesis, we get $F(\bd b_1)\ge n-3$. Hence,
    \[
        F(\bd a)\ \ge\ F(\bd b_0)+F(\bd b_1)\ \ge\ 2(n-3)\ \ge\ n-2\qquad (n\ge4).
    \]
    
    On the other hand if $a_{n-1}=1$, then $a_n=1$ as well. The hypothesis of Case A gives $a_1\le d-2$. We cannot have $n=4$. Indeed, the sorting condition gives $a_1\ge a_2$ and
    $a_1+a_2=2d-2$, so $a_2=2d-2-a_1\ge 2d-2-(d-2)=d\ge a_1+2>a_1$, a contradiction. Therefore,
    $n\ge5$. The summand for $t=1$ has $j_1=0$. Hence, by Lemma~\ref{lem:T0} we may delete that letter. Therefore, we get the content $(a_1,\dots,a_{n-2})$ with half-sum $d-1$, all entries in $[1,d-2]$, i.e.\ a nondegenerate content for $(2,n-2,d-1)$ with $n-2\ge3$ letters. By induction hypothesis, it contributes $\ge (n-2)-2=n-4$. Hence,
    \[
        F(\bd a)\ \ge\ (n-3)+(n-4)\ =\ 2n-7\ \ge\ n-2\qquad(n\ge5).
    \]
    \noindent\textbf{Case B: $a_1+a_n\ge d$.} Put
    \[
        T:=d-1-a_1 .
    \]
    Since $\bd a$ is nondegenerate content, we get $T\ge0$. Rearranging the Case~B hypothesis gives $T\le a_n-1$. In particular, every $t$ with $0\le t\le T+1$ satisfies $t\le a_n=\min(a_{n-1},a_n)$, and hence contributes a summand in \eqref{eq:fusrecur}.
    
    We show that for each $0\le t\le T$ the corresponding summand in \eqref{eq:fusrecur} contributes $\ge n-3$. Indeed, fix such a $t$. The summand has content $\bd b_t=(a_1,\dots,a_{n-2},j_t)$ and half-sum $d-t$. We check nondegeneracy.
    
    First, we check the lower bound on $j_t$: since $t\le T$,
        \[
            j_t\ \ge\ j_T=a_{n-1}+a_n-2(d-1-a_1)=(a_1+a_{n-1})+(a_1+a_n)-2d+2\ \ge\ d+d-2d+2=2,
        \]
    using $a_1+a_{n-1}\ge a_1+a_n\ge d$ twice. The other entries are $\ge1$ by hypothesis. Next, we check the upper bound on the $a_i$: $a_i\le a_1=d-1-T\le d-1-t=(d-t)-1$. Finally, we check the upper bound on $j_t$: $j_t\le(d-t)-1$ is equivalent to $a_{n-1}+a_n\le d-1+t$. For $t\ge1$ this follows from~\eqref{eq:pairA2}. For $t=0$ it follows from~\eqref{eq:pairA2} unless $n=4$ and all $a_i$ are equal. We handle this exceptional case separately below. Therefore, apart from the exceptional case, each such $\bd b_t$ is a nondegenerate content for $(2,n-1,d-t)$, and contributes $F(\bd b_t)\ge n-3$ by induction hypothesis.
    
    Suppose $n=4$ and $a_1=a_2=a_3=a_4=:a$, so $d=2a$.
    Then the summand for $t=0$ has content $(a,a,2a)$ with half-sum $2a$. Its last entry equals its half-sum, so $F=1$ by Lemma~\ref{lem:star}. Since $n-3=1$, we see that the summand still contributes at least $n-3$, as required. For $1\le t\le T=a-1$ the arguments above apply verbatim.
    
    Now consider $t=T+1=d-a_1\le a_n$. Its content is $\bd b_{t}=(a_1,\dots,a_{n-2},j_{t})$ with $j_{t}=j_T-2\ge0$ and half-sum $d-t=a_1$. The entry $a_1$ equals the half-sum, so $F(\bd b_{t})=1$ by Lemma~\ref{lem:star}.
    
    Adding up the $T+1$ summands with $0\le t\le T$ and the summand for $t=T+1$,
    \[
        F(\bd a)\ \ge\ (T+1)(n-3)+1\ \ge\ (n-3)+1\ =\ n-2 ,
    \]
    using $T+1\ge1$ and $n-3\ge1$.
\end{proof}

\subsection{Grassmannian \texorpdfstring{$\Gr(3,6)$}{Gr(3,6)}}

For rectangular shapes there is a second description of SSYT which will be used in the proof of the main result of this subsection.

\begin{defn}
\label{def:subsetorder}
    For $k$-element subsets $A,B\subseteq\{1,\dots,n\}$, written in increasing order as
    $A=\{\alpha_1<\alpha_2<\cdots<\alpha_k\}$ and $B=\{\beta_1<\cdots<\beta_k\}$, define
    \[
        A\le B\quad\Longleftrightarrow\quad \alpha_i\le\beta_i\ \text{ for all }i=1,\dots,k .
    \]
    This is a partial order.
\end{defn}

\begin{defn}
    A \emph{multichain} of length $d$ is a sequence $C_1\le C_2\le\cdots\le C_d$ of $k$-subsets
    of $\{1,\dots,n\}$ (repetitions allowed). Its
    \emph{content} is $\bd a=(a_1,\dots,a_n)$, where $a_i=\#\{j: i\in C_j\}$ is the number of
    members of the chain that contain the letter $i$.
\end{defn}

\begin{lem}
\label{lem:columnmodel}
    Let $\lam=(d^k)$. The map sending an SSYT $T$ to the sequence of its columns
    $C_1,\dots,C_d$, read as subsets of $\{1,\dots,n\}$, is a bijection
    \[
        \SSYT\bigl((d^k),\bd a\bigr)\ \xrightarrow{\ \sim\ }\
        \{\text{multichains } C_1\le\cdots\le C_d \text{ of $k$-subsets with content }\bd a\}.
    \]
\end{lem}

For two partitions $\lambda=(\lambda_1,\dotsc,
\lambda_m)$ and $\mu=(\mu_1,\dotsc,\mu_m)$ of an integer $\ell\in\mathbb{N}$, we say $\lambda$ \emph{dominates} $\mu$, written as $\lambda\dom\mu$, if and only if $\lambda_1+\dotsc+\lambda_i\geq \mu_1+\dotsc+\mu_i$ for all $1\leq i\leq m$. Recall the Kostka positivity condition.

\begin{theorem}[see Theorem 1 in \cite{KostPos1}, or \cite{KostPos2}]
\label{thm:positivity}
    Let $\lam$ be a partition and $\bd a=(a_1\ge a_2\ge\cdots\ge a_m)$ a weakly decreasing
    sequence of non-negative integers with $\lvert\lam\rvert=\lvert\bd a\rvert$. If $\lam\dom\bd a$, then
    $\Kos{\lam}{\bd a}\ge1$.
\end{theorem}

We need the following proposition to prove the main result of this subsection.

\begin{prop}
\label{prop:Tstar}
    Let $\bd a=(a_1,a_2,a_3,a_4,a_5,a_6)$ be a nondegenerate content for $(3,6,d)$. If $a_5+a_6\leq d-1$, then there exists $T^\ast\in\SSYT((d^3),\bd a)$ having a column that contains both the letters $5$ and $6$.
\end{prop}

\begin{proof}
    Without loss of generality we assume $a_1\ge a_2\ge\cdots\ge a_6\ge1$. Since $\bd a$ is nondegenerate content and $a_5+a_6\leq d-1$, notice that $d\geq 3$. Set
    \[
        \mu\ :=\ \bigl(d,\ d-1,\ d-a_5-a_6+1\bigr).
    \]
    We have $d\ge d-1$, and $d-1\ge d-a_5-a_6+1$ because
    $a_5+a_6\ge2$. Moreover, since $a_5+a_6\le d-1$, we have $\mu_3=d-a_5-a_6+1\ge2$. Hence, we conclude that $\mu$ is a partition with three rows.
    
    We check the dominance condition $\mu\dom(a_1,a_2,a_3,a_4)$. For this we check the three nontrivial inequalities.
    \[
        \mu_1=d\ \ge\ a_1 ,\qquad
        \mu_1+\mu_2=2d-1\ \ge\ a_1+a_2 ,\qquad
        \mu_1+\mu_2+\mu_3=|\mu|\ \ge\ a_1+a_2+a_3
    \]
    Using the fact that $a_1\le d-1$, $a_1+a_2\le 2d-2$, and
    $\lvert\mu\rvert=a_1+\cdots+a_4\ge a_1+a_2+a_3$ gives the first, second, and third inequalities above, respectively. The fourth condition is the equality of sizes: we have $\lvert\mu\rvert=3d-a_5-a_6=a_1+a_2+a_3+a_4$. 
    
    By Theorem~\ref{thm:positivity} there is a semistandard tableau $U$ of shape $\mu$ in the
    letters $\{1,2,3,4\}$ with content $(a_1,a_2,a_3,a_4)$. Now define $T^\ast$ of shape
    $(d^3)$ in the following way:
    \begin{enumerate}[nosep]
        \item on the cells of $\mu$, put the entries of $U$;
        \item in cell $(2,d)$, put $5$;
        \item in the cells $(3,\,d-a_5-a_6+2),\ \dots,\ (3,\,d-a_6)$, put $5$. There are $a_5-1$ such cells;
        \item in the cells $(3,\,d-a_6+1),\ \dots,\ (3,d)$, put $6$. There are $a_6$ such cells.
    \end{enumerate}
    We check the row lengths. Row $1$ has $\mu_1=d$ cells; row $2$ has $\mu_2+1=d$; row $3$ has
    $\mu_3+(a_5-1)+a_6=d$. So the shape is the rectangle $(d^3)$, and the content is
    $(a_1,a_2,a_3,a_4,a_5,a_6)=\bd a$.
    
    We check the semistandardness of rows. Row $1$ is a row of $U$. Row $2$ is a row of $U$ (entries
    $\le4$) followed by a $5$. Row $3$ is a row of $U$ (entries $\le4$), then $a_5-1$ copies of
    $5$, then $a_6$ copies of $6$. All weakly increasing.
    
    We check the semistandardness of columns. Let $j$ be a column index. If $j\le\mu_3$, all three cells of column $j$ belong to $\mu$ and carry the entries of $U$, which satisfies the strictly increasing column condition. If $\mu_3<j\le d-1$, then rows $1$ and $2$ of column $j$ lie in $\mu$ (indeed $j\le d-1=\mu_2\le\mu_1$) and carry entries $\le4$, strictly increasing; row $3$ carries $5$ or $6$, which is larger. Finally, if $j=d$, the column reads $U_{1,d}$, then $5$, then $6$. Since $U_{1,d}\le4<5<6$, the column is strictly increasing.
    
    Finally, column $d$ contains both $5$ (in row $2$) and $6$ (in row $3$).
\end{proof}

We present the main theorem of this subsection.

\begin{theorem}\label{thm:B}
    Every nondegenerate content $\bd a$ for $(3,6,d)$ satisfies $\Kos{(d^3)}{\bd a}\ge4$.
\end{theorem}

\begin{proof}
    Without loss of generality we assume $a_1\ge a_2\ge\cdots\ge a_6\ge1$. We consider the following two cases.
    
    \noindent\textbf{Case A: $a_5+a_6\le d-1$.} The merged content $\bd a^{\,\mathrm{mrg}}=(a_1,a_2,a_3,a_4,\,a_5+a_6)$ is a nondegenerate content for $(3,5,d)$. Indeed, the merged content has five entries summing to $3d$. Its first four entries lie in $[1,d-1]$ since $\bd a $ is a nondegenerate content for $(3,6,d)$, and the fifth entry is $a_5+a_6$, which is $\ge2$ and $\le d-1$ by hypothesis. So it is a nondegenerate content for $(3,5,d)$. Therefore, we get
    \[
        \Kos{(d^3)}{\bd a}\ \ge\ \Kos{(d^3)}{\bd a^{\,\mathrm{mrg}}}\ \ge\ 3 .
    \]
    The first inequality follows from Lemma~\ref{lem:polarization}, and by using duality of $\Gr(2,5)$ and $\Gr(3,5)$ and Proposition~\ref{prop:A}, we get $\Kos{(d^3)}{\bd a^{\,\mathrm{mrg}}}\ge3$. 

    Therefore, by Proposition~\ref{prop:T1image} the image of the merging injection $\Phi$ contains at
    least $3$ tableaux, all of which have no column containing both $5$ and $6$. By Proposition~\ref{prop:Tstar} there is at least one more tableau of the same content which does have such a column, and hence lies outside the image of $\Phi$. Therefore, $\Kos{(d^3)}{\bd a}\ge 3+1=4$.

    \noindent\textbf{Case B: $a_5+a_6\ge d$.} We have $(k,n)=(3,6)$, so $n=2k$. By using Lemma~\ref{lem:pairbound} and Corollary~\ref{cor:pairbound}, we conclude $d$ is even and $a_1=\cdots=a_6=d/2$. We write $d=2a$ with $a\geq 1$, and $\bd a=(a,a,a,a,a,a)$.

    We will use Lemma~\ref{lem:columnmodel} and exhibit four multichains. For a pair $(A,B)$ of
    complementary $3$-subsets of $\{1,\dots,6\}$ with $A\le B$, the sequence
    \[
        \underbrace{A\le A\le\cdots\le A}_{a}\ \le\ \underbrace{B\le B\le\cdots\le B}_{a}
    \]
    is a multichain of length $2a=d$, and each letter lies in exactly $a$ of its members, namely in the $a$ copies of whichever of $A,B$ contains it. So its content is $(a^6)=\bd a$. Distinct $A$'s give distinct multichains.
    
    We have at least four valid choices:
    \[
        \begin{array}{llll}
            A=\{1,2,3\},&B=\{4,5,6\}:&1\le4,\ 2\le5,\ 3\le6\\
            A=\{1,2,4\},&B=\{3,5,6\}:&1\le3,\ 2\le5,\ 4\le6\\
            A=\{1,2,5\},&B=\{3,4,6\}:&1\le3,\ 2\le4,\ 5\le6\\
            A=\{1,3,4\},&B=\{2,5,6\}:&1\le2,\ 3\le5,\ 4\le6
        \end{array}
    \]
    Hence $\Kos{(d^3)}{\bd a}=\Kos{(d^3)}{(a^6)}\ge4$.
\end{proof}

\section{Main result}
\label{sec:section 6}

In this section we obtain a lower bound on the dimension of the weight spaces of $\dim_{\mathbb{K}}\SI^{\,\Gr(k,n)}_{(d;\bd a)}$ which is needed for proving Conjecture \ref{conj:three} for the rectangular Kostka coefficients.

Let
\[
    D:=\{(k,n):k\ge2,\ n\ge k+2\},\qquad
    E:=\{(2,4),(2,5),(3,5)\},\qquad
    N:=D\setminus E
\]
We order $N$ lexicographically, i.e., $(k',n')<_{\mathrm{lex}}(k,n)$ means $k'<k$, or
$k'=k$ and $n'<n$.

\begin{lem}\label{lem:wellfounded}
    The lexicographic order on $N$ is well-founded.
\end{lem}

\begin{proof}
    Let $(k_1,n_1)>(k_2,n_2)>\cdots$ be a strictly decreasing chain. The integers $k_i$ are
    non-increasing and bounded below by $2$, and hence has to be eventually constant. Say $k_i=k$ for $i\ge i_0$. For $i\ge i_0$ the integers $n_i$ has to be strictly decreasing while remaining bounded below by $k+2$, contradiction.
\end{proof}

\begin{theorem}\label{thm:star}
    Let $k\ge2$, $n\ge k+2$, $d\ge1$, and let $\bd a$ be a nondegenerate content for $(k,n,d)$, i.e.\ $\sum_ia_i=kd$ and $1\le a_i\le d-1$ for all $i$. We obtain the following bounds on the weight spaces determined by the nondegenerate content $\bd a$
    \[
        \dim_{\mathbb{K}}\SI^{\,\Gr(k,n)}_{(d;\bd a)}\ =\ \Kos{(d^k)}{\bd a}\ \ge\
        \begin{cases}
            2,&(k,n)=(2,4),\\[1mm]
            3,&(k,n)\in\{(2,5),\,(3,5)\},\\[1mm]
            4,&\text{otherwise}
        \end{cases}
    \]
\end{theorem}

\begin{proof}
    Let the sets $D,E$, and $N$ be as above. For $(k,n)\in E$, the result follows from Proposition \ref{prop:A}, and by using the duality of the Grassmannians $\Gr(2,5)$ and $\Gr(3,5)$.
    
    For $(k,n)\in N$, let $\mathcal{P}(k,n)$ be the statement: for every $d\ge1$ and every nondegenerate content $\bd a$ for $(k,n,d)$ one has $\Kos{(d^k)}{\bd a}\ge4$. In order to prove the statement $\mathcal{P}$ on all of $N$, by Lemma~\ref{lem:wellfounded} and well-founded induction, it suffices to prove that for each $x\in N$, if $\mathcal{P}(y)$ holds for every $y\in N$ with $y<_{\mathrm{lex}}x$, then $\mathcal{P}(x)$ holds.
    
    Fix $x=(k,n)\in N$ and assume
    \begin{equation*}
        \mathcal{P}(k',n')\quad\text{for every }(k',n')\in N\text{ with }(k',n')<_{\mathrm{lex}}(k,n).
    \end{equation*}
    Let $d\ge1$ and let $\bd a$ be a nondegenerate content for $(k,n,d)$. Without loss of generality we may assume $a_1\ge a_2\ge\cdots\ge a_n$. We will show
    $\Kos{(d^k)}{\bd a}\ge4$ by considering the following cases:
    
    Let $k=2$. Since $(2,4),(2,5)\notin N$, we have $n\ge6$. Proposition~\ref{prop:A} gives
    \[
        K_{(d,d),\mathbf a}\ge n-2\ge4.
    \]
    
    Let $k=3$. If $(k,n)=(3,6)$, then Theorem~\ref{thm:B} gives $K_{(d^3),\mathbf a}\ge4$. Therefore, assume $n\ge7$. By Corollary~\ref{cor:pairbound},
    \[
        a_{n-1}+a_n\le \frac{6d}{n}\le\frac{6d}{7}<d
    \]
    so $a_{n-1}+a_n\le d-1$. Hence, $\mathbf a^{\mathrm{mrg}}$ has $a_i\in[1,d-1]$ for $i\leq n-2$, and merged entry $a_{n-1}+a_n\in[2,d-1]$. We see that it is a nondegenerate content for $(3,n-1,d)$. Since $n-1\geq 6$, we have $(3,n-1)\in N$. Moreover $(3,n-1)<_{\text{lex}}(3,n)$. Therefore, by Lemma~\ref{lem:polarization} and induction hypothesis, we get
    \[
        K_{(d^3),\mathbf a}\ge K_{(d^3),\mathbf a^{\mathrm{mrg}}}\ge4.
    \]
    For the rest of the proof we assume $k\geq 4$. We will consider $n$ in different ranges: $n\ge2k+1$, $n=2k$, and $k+2\le n\le2k-1$.
    
    Suppose $n\ge2k+1$. Identical to the case above, with
    \[
        a_{n-1}+a_n\le\frac{2kd}{n}\le\frac{2kd}{2k+1}<d
    \]
    we see that $\mathbf a^{\mathrm{mrg}}$ is a nondegenerate content for $(k,n-1,d)$. We have $(k,n-1)$ in $N$ because $n-1\ge2k\ge k+2$ and $k\ge4$ excludes $E$, and is lexicographically smaller. So
    \[
        K_{(d^k),\mathbf a}\ge K_{(d^k),\mathbf a^{\mathrm{mrg}}}\ge4.
    \]
    
    Next, we suppose $n=2k$. If $a_{n-1}+a_n\le d-1$, then merging gives a nondegenerate content for $(k,2k-1,d)$. We have $(k,2k-1)\in N$ because $2k-1\ge k+2$ ($k\ge3$) and $k\ge4$ and is lexicographically smaller. So
    \[
        K_{(d^k),\mathbf a}\ge K_{(d^k),\mathbf a^{\mathrm{mrg}}}\ge4.
    \]
    On the other hand, if $a_{n-1}+a_n\ge d$, then by Corollary~\ref{cor:pairbound} we get that $d$ is even and
    \[
        \mathbf a=\big((\tfrac d2)^{2k}\big)
    \]
    with $d\ge2$. If we merge the last two letters, then we get the content $\bigl((d/2)^{2k-2},d\bigr)$, which has an entry equal to $d$. Using Lemma~\ref{lem:polarization} and Lemma~\ref{lem:T3}, we get
    \[
        \Kos{(d^k)}{\bd a}\ \ge\ \Kos{(d^k)}{\bigl((d/2)^{2k-2},d\bigr)}
        \ =\ \Kos{(d^{\,k-1})}{\bigl((d/2)^{2k-2}\bigr)} .
    \]
    The content $\bigl((d/2)^{2k-2}\bigr)$ has $2k-2$ entries summing to $(k-1)d$, each equal
    to $d/2\in[1,d-1]$. Therefore, it is a nondegenerate content for $(k-1,2k-2,d)$. Moreover, $(k-1,2k-2)\in N$; indeed $2k-2\ge(k-1)+2$ since $k\ge3$, and $(k-1,2k-2)\notin E$ since $k-1\ge3$. In the boundary case $k-1=3$, one has $2k-2=6\ne5$. It is lexicographically smaller since $k-1<k$. Hence, the induction hypothesis gives $\ge4$.
    
    Finally, we suppose $k+2\le n\le2k-1$. By Lemma~\ref{lem:T2},
    \[
        \Kos{(d^k)}{\bd a}\ =\ \Kos{(d^{\,n-k})}{\bd a^{\vee}}\quad \text{where}
        \quad \bd a^\vee:=(d-a_1,\dots,d-a_n)
    \]
    and $\bd a^\vee$ is nondegenerate for $(n-k,n,d)$. We have $(n-k,n)\in N$. Indeed, $n-k\ge2$ because $n\ge k+2$, and $n\ge(n-k)+2$ because $k\ge2$. Additionally, since $n\ge k+2\ge6$, we have $(n-k,n)\notin E$. Since $n-k\le k-1<k$, it is lexicographically smaller. Hence, the induction hypothesis gives $\ge4$.
\end{proof}

\begin{remark}
\label{rem:trivialcase}
    In our setting, the relevant range of parameters is $k\geq 2$ and $n\geq k+2$. In all other cases, we obtain $\dim_{\mathbb{K}}\SI^{\,\Gr(k,n)}_{(d;\bd a)}\ =\ \Kos{(d^k)}{\bd a}\leq1$. Indeed, for $k = 1$ a tableau with one row is determined by its content. For $n = k$, the only filling of $(d^k)$ has row $i$ filled with the letter $i$. For $n = k+1$ use duality of Grassmannians $\Gr(k,n)\cong\Gr(n-k,n)$ to reduce to $k=1$ case.
\end{remark}

\section{King, Tollu, Toumazet conjecture for rectangular Kostka coefficients}
\label{sec:section 7}

A conjecture of King, Tollu, Toumazet \cite[see Conj. 3.3]{soton41202} states that for partitions $\lambda,\mu$ and $\nu$, if the Littlewood-Richardson coefficient $c^{\nu}_{\lambda,\mu}=3$, then $c^{M\nu}_{M\lambda,M\mu}=\binom{M+2}{2}$ or $2M+1$ for every positive integers $M$. In this section we will prove this conjecture for rectangular Kostka coefficients, which are a special case of Littlewood-Richardson coefficients.

\subsection{Grassmannian \texorpdfstring{$\Gr(2,5)$}{Gr(2,5)}}

We will first prove the King, Tollu, Toumazet conjecture for $\Gr(2,5)$ by using the idea of Gelfand-Tsetlin polytope.

Let $T$ be a semistandard tableau of shape $(d,d)$ with entries in $\{1,\dots,5\}$
and content $\bd a$. Let $x_i$ be the number of $i$'s in the first row and set
\[
X_j:=x_1+\cdots+x_j,\qquad A_j:=a_1+\cdots+a_j,\qquad X_0=A_0=0 .
\]
The second row then contains $y_i=a_i-x_i$ copies of $i$, with partial sums
$Y_j=A_j-X_j$.
 
\begin{lem}
\label{lem:coords}
    The assignment $T\mapsto(X_1,\dots,X_4)$ is a bijection from the set of such
    tableaux onto the set of integer vectors satisfying
    \begin{equation}\label{eq:cons}
        X_5=d,\qquad 0\le X_j-X_{j-1}\le a_j\ \ (1\le j\le 5),\qquad X_{j-1}+X_j\ \ge\ A_j\ \ (1\le j\le5).
    \end{equation}
\end{lem}
 
\begin{proof}
    A tableau of shape $(d,d)$ is determined by its two rows, and each row is determined
    by its multiset of entries since rows are weakly increasing. So $T$ is determined by
    $(x_1,\dots,x_5)$, equivalently by $(X_1,\dots,X_5)$ with $X_5=d$; the conditions
    $0\le x_j\le a_j$ are the middle inequalities in \eqref{eq:cons}.
     
    Now we will check the column strictness condition. The entry of row $1$ in column $p$ of the tableau $T$ is $\min\{j:X_j\ge p\}$ and the entry of row $2$
    in column $p$ is $\min\{j:Y_j\ge p\}$. Column strictness means that the row-$2$ entry exceeds the row-$1$ entry in every column. For fixed $j$, the columns whose row-$2$ entry is $\le j$ are exactly
    $p=1,\dots,Y_j$ and for each of these we need the row-$1$ entry to be $\le j-1$, i.e.,
    $p\le X_{j-1}$. Hence column strictness is equivalent to $Y_j\le X_{j-1}$ for all $j$,
    which is $A_j-X_j\le X_{j-1}$. This is the last family of inequalities in \eqref{eq:cons}. Conversely, these inequalities give column strictness.
\end{proof}

\begin{remark}
    Notice that two of the four coordinates are fixed. Indeed, taking $j=1$ in \eqref{eq:cons} gives $X_1\ge a_1$, while $X_1=x_1\le a_1$; so $X_1=a_1$. Taking $j=5$ gives $X_4\ge A_5-X_5=2d-d=d$, while $X_4\le X_5=d$; so $X_4=d$. Since $X_5=d$, we conclude that $x_5=0$, i.e.\ no $5$'s
    occur in the first row. Therefore, there are only two free entries
    \[
        u:=X_2,\; v:=X_3 .
    \]
\end{remark}

We use the above inequalities to define a polytope whose number of lattice points gives the dimension of the weight spaces $\dim_{\mathbb{K}}\SI^{\,\Gr(2,5)}_{(d;\bd a)}$. Substituting $X_1=a_1$, $X_4=X_5=d$ into \eqref{eq:cons} and writing $u=X_2$,
$v=X_3$ turns the above inequalities into exactly (C1)--(C9) below.
 
\begin{defn}
\label{def:polygon}
    For a content $\bd a$ with $\sum a_i=2d$ let $\mathcal P_{\bd a}\subset\mathbb{R}^2$ be the
    polyhedron cut out by
    \[
        \begin{array}{llll}
            \text{(C1)}\ u\ge a_1, &\text{(C2)}\ u\ge a_2, &\text{(C3)}\ u\le a_1+a_2, &
            \text{(C4)}\ u+v\ge A_3,\\[1mm]
            \text{(C5)}\ v\ge u, &\text{(C6)}\ v\le u+a_3, &\text{(C7)}\ v\ge d-a_4, &
            \text{(C8)}\ v\ge d-a_5,\\[1mm]
            \text{(C9)}\ v\le d. & & &
        \end{array}
    \]
\end{defn}
 
\begin{theorem}
\label{thm:integral}
    For every content $\bd a$ of $(2,5,d)$, $\mathcal P_{\bd a}$ is a lattice polytope.
\end{theorem}
 
\begin{proof}
    Every extreme point $p$ of $\mathcal P_{\bd a}$ is the unique solution of a pair of
    constraints from (C1)--(C9) whose normal vectors are linearly independent.
     
    The normal vectors occurring in Definition~\ref{def:polygon} are
    \begin{gather*}
        (1,0)\ \ \text{(C1,C2)},\qquad(-1,0)\ \ \text{(C3)},\qquad(1,1)\ \ \text{(C4)},\qquad(-1,1)\ \ \text{(C5)},\\[1mm]
        (1,-1)\ \ \text{(C6)},\qquad(0,1)\ \ \text{(C7,C8)},\qquad(0,-1)\ \ \text{(C9)},
    \end{gather*}
    and every right-hand side is an integer. If a pair of independent normals has determinant $\pm1$, then by Cramer's rule we get an integral solution. Inspecting all pairs, we get determinant of absolute value $2$ only for the following two pairs
    \[
        \{(1,1),(-1,1)\}=\{\text{C4},\text{C5}\}
        \qquad\text{and}\qquad
        \{(1,1),(1,-1)\}=\{\text{C4},\text{C6}\}
    \]
    All other pairs hav determinant $\pm1$ or are parallel.
     
    Consider the pair {\rm(C4),(C6)}. Solving $u+v=A_3$ and $v-u=a_3$ gives
    \[
        u=\tfrac{A_3-a_3}{2}=\tfrac{a_1+a_2}{2},\qquad v=\tfrac{A_3+a_3}{2}.
    \]
    If this point lies in $\mathcal P_{\bd a}$ then (C1) gives $\tfrac{a_1+a_2}{2}\ge a_1$, i.e., $a_2\ge a_1$, and (C2) gives $\tfrac{a_1+a_2}{2}\ge a_2$, i.e., $a_1\ge a_2$. Hence $a_1=a_2$. Therefore, $u=a_1\in\ZZ$ and $v=a_1+a_3\in\ZZ$.
     
    Now consider the pair {\rm(C4),(C5)}. Solving $u+v=A_3$ and $v=u$ gives $u=v=\tfrac{A_3}{2}$. Since $2d=A_3+a_4+a_5$ we have $d-a_4=\tfrac{A_3+a_5-a_4}{2}$ and $d-a_5=\tfrac{A_3+a_4-a_5}{2}$. If the point lies in $\mathcal P_{\bd a}$ then (C7) gives $\tfrac{A_3}{2}\ge\tfrac{A_3+a_5-a_4}{2}$,
    i.e., $a_4\ge a_5$, and (C8) gives $a_5\ge a_4$. Hence $a_4=a_5$. So $A_3=2d-2a_4$ is even and $u=v=d-a_4\in\ZZ$.
    
    In both cases the point is integral whenever it belongs to $\mathcal P_{\bd a}$, or otherwise it is not a vertex. Hence all vertices are integral.
\end{proof}

\begin{cor}
\label{cor:KTTforGr(2,5)}
    Let $\bd a$ be a content for $(2,5,d)$ with $d\geq 1$. If $\dim_{\mathbb{K}}\SI^{\,\Gr(2,5)}_{(d;\bd a)}=3$, then $\dim_{\mathbb{K}}\SI^{\,\Gr(2,5)}_{(Md;M\bd a)}=\binom{M+2}{2}$ or $2M+1$, for every positive integer $M$.
\end{cor}

\begin{proof}
    Since $\dim_{\mathbb{K}}\SI^{\,\Gr(2,5)}_{(d;\bd a)}=3$, the lattice polytope $\mathcal{P}_{\bd a}$ is either a line segment or a triangle with $3$ lattice points. The result follows.
\end{proof}

\subsection{Grassmannian \texorpdfstring{$\Gr(k,n)$}{Gr(k,n)}}

Now we will prove the King, Tollu, Toumazet conjecture for all Grassmannians $\Gr(k,n)$.

\begin{theorem}
\label{thm:KTT}
    Let $k\ge2$, $n\ge k+2, d\geq 1$, and $\bd a$ be a content for $(k,n,d)$. If $\dim_{\mathbb{K}}\SI^{\,\Gr(k,n)}_{(d;\bd a)}=3$, then $\dim_{\mathbb{K}}\SI^{\,\Gr(k,n)}_{(Md;M\bd a)}=\binom{M+2}{2}$ or $2M+1$, for every positive integer $M$.
\end{theorem}

\begin{proof}
    Let $Q$ be the subspace quiver with $n$ arms and $\beta=(k,1,1,\dots,1)$, so that $\SI(Q,\beta)$ is the homogeneous coordinate ring of $\Gr(k,n)$, and let $\sigma=(-d,\,a_1,\dots,a_n)$ be the weight associated with the content $\bd a$. Since $\dim_{\mathbb{K}}\SI(Q,\beta)_\sigma=3\ne0$ we have $\sigma\in\Sigma(Q,\beta)$, i.e.\ $\beta$ is $\sigma$-semistable. Write
    \[
        \sigma=e\,\sigma_0,\qquad e:=\gcd\bigl(d,a_1,\dots,a_n\bigr),
    \]
    so that $\sigma_0$ is an indivisible weight. Let
    \[
        \beta=c_1\cdot\alpha_1\dotplus c_2\cdot\alpha_2\dotplus\cdots\dotplus c_r\cdot\alpha_r
    \]
    be the $\sigma_0$-stable decomposition of $\beta$, with the $\alpha_i$ distinct. Since $\beta=(k,1,1,\dots,1)$, we conclude that all $c_i=1$ and
    \[
        \beta=\alpha_1\dotplus\cdots\dotplus\alpha_r,\quad\text{where}\quad\alpha_i=(k_i,\mathbf 1_{S_i}),\quad\text{with}\quad
        \bigsqcup_{i}S_i=\{1,\dots,n\}\quad\text{and}\quad \sum_ik_i=k
    \]
    Moreover, notice that $\SI(Q,\alpha_i)$ is the homogeneous coordinate ring of $\Gr(k_i,|S_i|)$. The weight is the restriction of $\sigma$ itself, whose value at the central vertex is still $-d$. Hence, by Proposition~\ref{prop:dimformula} the
    number of rows is $k_i$, the number of letters is $|S_i|$, and the content is the restriction $\bd a|_{S_i}=(a_j)_{j\in S_i}$. Using $\sigma_0(\alpha_i)=0$ we get
    \[
        \sum_{j\in S_i}a_j=k_i\,d 
    \]
    Additionally, $a_j\leq d$. So $\bd a|_{S_i}$ is indeed a content for $(k_i,|S_i|,d)$.

    Applying \cite[Theorem 3.20]{CombQuivDerkWey} to the indivisible weight $\sigma_0$ and using $c_i=1$, we get
    \[
        \SI(Q,\beta)_{m\sigma_0}\;\cong\;\bigotimes_{i=1}^{r}\mathrm{Sym}^{c_i}\bigl(\SI(Q,\alpha_i)_{m\sigma_0}\bigr)
        \;=\;\bigotimes_{i=1}^{r}\SI(Q,\alpha_i)_{m\sigma_0}\qquad(m\ge0).
    \]
    Substituting $m=eM$ gives, for every integer $M\ge1$,
    \begin{equation}\label{eq:prod}
        \dim_{\mathbb{K}}\SI(Q,\beta)_{M\sigma}\;=\;\prod_{i=1}^{r}\dim_{\mathbb{K}}\SI(Q,\alpha_i)_{M\sigma}.
    \end{equation}
        Take $M=1$. Since $3$ is prime, we get that there is exactly one index $i$ with $\dim_{\mathbb{K}}\SI(Q,\alpha_{i})_\sigma=3$, all other factors having dimension $1$, and by Fulton's property \cite[Appendix]{CombQuivDerkWey} these factors stay $1$-dimensional on scaling by $M$. Write $\alpha_{i}=(k',\mathbf 1_{S})$ with $S\coloneqq S_{i}$ and $n'\coloneqq\lvert S\rvert$. Since $\dim_{\mathbb{K}}\SI(Q,\alpha_{i})_\sigma=3>1$, we conclude that $k'\geq 2$ and $n'\geq k'+2$ by Remark \ref{rem:trivialcase}. Additionally, since $\alpha_i$ is $\sigma_0$-stable, and hence $\sigma$-stable, by using Proposition~\ref{prop:stability}(b)(ii) we get
    \[
        1\le a_j\le d-1\qquad\text{for all }j\in S ,
    \]
    i.e.\ $\bd a|_{S}$ is a nondegenerate content for $(k',n',d)$.

    We have $\dim_{\mathbb{K}}\SI(Q,\alpha_{i})_\sigma=3<4$. Therefore, by using Theorem \ref{thm:star} we conclude that
    \[
        (k',n')\in\bigl\{(2,4),\,(2,5),\,(3,5)\bigr\}.
    \]
    The result is true for $\Gr(2,5)$ by Corollary~\ref{cor:KTTforGr(2,5)}, and for $\Gr(3,5)$ via duality of Grassmannians $\Gr(2,5)$ and $\Gr(3,5)$. In this case, we get the dimension of the scaled weight space as $\binom{M+2}{2}$. If $(k',n')=(2,4)$, then the dimension vector $\alpha_i$ is an isotropic Schur root. Hence, as a consequence of \cite[Section 3.2 \& Proposition 3.2.6]{UFDPaqPraWeh}, we conclude that $\sigma=2\chi+\mathrm{rem}(\sigma)$, and gives the result for $\Gr(2,4)$. In this case, we get the dimension of the scaled weight space as $2M+1$.
    
    Hence, for every integer $M\geq 1$ we get that
    \[
        \dim_{\mathbb{K}}\SI^{\,\Gr(k,n)}_{(Md;M\bd a)}\;=\;\dim_{\mathbb{K}}\SI(Q,\beta)_{M\sigma}\;=\;\binom{M+2}{2}\quad \text{or}\quad 2M+1
    \]
\end{proof}

\begin{remark}
    We believe that the methods in this paper can be generalized to prove the conjecture for all Kostka coefficient. This will be dealt with in the upcoming future work.
\end{remark}

Below we give some examples to illustrate the main idea of the proof.

\begin{example}
    Consider the $6$-subspace quiver
    \[
      \begin{tikzcd}[column sep=0.9em, row sep=3.5em]
        1 && 2 && 3 && 4 && 5 && 6 \\
        &&&&& 0
        \arrow[from=1-1,  to=2-6]
        \arrow[from=1-3,  to=2-6]
        \arrow[from=1-5,  to=2-6]
        \arrow[from=1-7,  to=2-6]
        \arrow[from=1-9,  to=2-6]
        \arrow[from=1-11, to=2-6]
      \end{tikzcd}
    \]
    with dimension vector $\beta=(3,1,1,1,1,1,1)$. The semi-invariant ring is isomorphic to the homogeneous coordinate ring of $\Gr(3,6)$. Let $\sigma=(-3,3,2,1,1,1,1)$. Here $d=3$ and $\bd a=(a_1,\dotsc,a_6)=(3,2,1,1,1,1)$. Notice that since $a_1=3=d$, $\beta$ is not $\sigma$-stable. Up to reordering, the $\sigma$-stable decomposition of $\beta$ is
    \[
        \beta=(1,1,0,0,0,0,0)\dotplus(2,0,1,1,1,1,1)
    \]
    The summand $\alpha_1\coloneqq(1,1,0,0,0,0,0)$ is a real Schur root. For $\alpha_2\coloneqq(2,0,1,1,1,1,1)=(2,\mathbf{1}_{S_2})$ where $S_2=\lbrace2,3,4,5,6\rbrace$, we have $\SI(Q,\alpha_2)\cong\mathbb{K}[\Gr(2,5)]$. We see that the content $\bd a\lvert_{S_2}=(2,1,1,1,1)$ is a nondegenerate content for $(k',n',d)=(2,5,3)$. Moreover, it can be checked that
    \[
        \Kos{(3,3)}{(2,1,1,1,1)}=3
    \]
    Hence, $\dim_{\mathbb{K}}\SI(Q,\beta)_{\sigma}=3$. Consequently, we get $\dim_{\mathbb{K}}\SI(Q,\beta)_{M\sigma}=\binom{M+2}{2}$ for every positive integer $M$.
\end{example}

\begin{example}
    Consider the $8$-subspace quiver
    \[
      \begin{tikzcd}[column sep=0.9em, row sep=3.5em]
        1 && 2 && 3 && 4 && 5 && 6 && 7 && 8 \\
        &&&&&&& 0
        \arrow[from=1-1,  to=2-8]
        \arrow[from=1-3,  to=2-8]
        \arrow[from=1-5,  to=2-8]
        \arrow[from=1-7,  to=2-8]
        \arrow[from=1-9,  to=2-8]
        \arrow[from=1-11, to=2-8]
        \arrow[from=1-13, to=2-8]
        \arrow[from=1-15, to=2-8]
      \end{tikzcd}
    \]
    with dimension vector $\beta=(4,1,1,1,1,1,1,1,1)$. The semi-invariant ring is isomorphic to the homogeneous coordinate ring of $\Gr(4,8)$. Let $\sigma=(-4,4,4,2,2,2,2,0,0)$. Here $d=4$ and $\bd a=(a_1,\dotsc,a_8)=(4,4,2,2,2,2,0,0)$. Let $\sigma=2\sigma_0$ with $\sigma_0=(-2,2,2,1,1,1,1,0,0)$, and $d_0\coloneqq d/2=2$. Hence, corresponding to $\sigma_0$, we get a content $\bd a'=(2,2,1,1,1,1,0,0)$ for $(4,8,2)$. Up to reordering, the $\sigma_0$-stable decomposition of $\beta$ is
    \[
        \beta=(1,e_1)\dotplus(1,e_2)\dotplus(2,\mathbf{1}_{\lbrace3,4,5,6\rbrace}) \dotplus(0,e_7)\dotplus(0,e_8)
    \]
    Let $\alpha_3\coloneqq(2,\mathbf{1}_{\lbrace3,4,5,6\rbrace})$. We have $\SI(Q,\alpha_3)\cong\mathbb{K}[\Gr(2,4)]$. The other dimension vectors in the above decomposition are real Schur roots.
    
    We see that $(1,1,1,1)$ is a nondegenerate content for $(k',n',d_0)=(2,4,2)$. In particular, $\bd a|_{\mathbf{1}_{\lbrace3,4,5,6\rbrace}}=2\times(1,1,1,1)=(2,2,2,2)$ is a nondegenerate content for $(k',n',d)=(2,4,4)$. Moreover, it can be checked that
    \[
        \Kos{(2,2)}{(1,1,1,1)}=2\quad\text{and}\quad \Kos{(4,4)}{(2,2,2,2)}=3
    \]
    Hence, $\dim_{\mathbb{K}}\SI(Q,\beta)_{\sigma}=3$. Consequently, we get $\dim_{\mathbb{K}}\SI(Q,\beta)_{M\sigma}=2M+1$ for every positive integer $M$.
\end{example}

\section*{Acknowledgment}
The author would like to thank Prof. Calin Chindris for suggesting this problem, and for many stimulating and fruitful discussions.

\section*{Disclosure of AI use}
Claude was used for editorial and paraphrasing purposes in various parts of the paper. The author takes responsibility for the mathematical content of the paper.

\bibliographystyle{alpha}
\bibliography{references}

\end{document}